\documentclass[11pt,reqno]{amsart}
\usepackage{amsbsy,amsfonts,amsmath,amssymb,amscd,amsthm}
\usepackage{mathrsfs,float}
\usepackage[mathcal]{euscript}
\usepackage{subfigure,enumitem,graphicx,epstopdf}
\usepackage[a4paper,text={6.1in,8in},centering]{geometry}
\usepackage{pxfonts,epstopdf}
\usepackage[colorlinks=true,linkcolor=blue,citecolor=green]{hyperref}
\usepackage{srcltx}
\usepackage[margin=20pt,font=small,labelfont=bf,textfont=it]{caption}

\hypersetup{linktocpage}

\small\normalsize
\theoremstyle{plain}
\newtheorem{mainthm}{Theorem}

\newtheorem{thm}{Theorem}[section]

\newtheorem{lem}[thm]{Lemma}

\newtheorem{prop}[thm]{Proposition}

\newtheorem{defi}[thm]{Definition}

\theoremstyle{definition}
\newtheorem{exap}[thm]{Example}
\newtheorem{rem}[thm]{Remark}

\newtheorem*{notation}{Notation}

\newcommand{\eqdef}{\stackrel{\scriptscriptstyle\rm def}{=}}

\makeatletter
\def\l@part{\@tocline{0}{-2pt}{1pc}{}{}}
\def\l@section{\@tocline{1}{-2pt}{1pc}{4.6em}{}}
\renewcommand{\tocpart}[3]{%
  \indentlabel{\@ifnotempty{#2}{\makebox[2.3em][l]{%
    \ignorespaces#1 #2.\hfill}}}\bf{#3}}
\renewcommand{\tocsection}[3]{%
  \indentlabel{\@ifnotempty{#2}{\hspace*{2.3em}\makebox[2.3em][l]{%
    \ignorespaces#1 #2.\hfill}}}#3}

\makeatother

\let\oldtocsection=\tocsection
\let\oldtocsubsection=\tocsubsection
\renewcommand{\tocsection}[2]{\hspace{0em}\bf\oldtocsection{#1}{#2}}
\renewcommand{\tocsubsection}[2]{\hspace{4.8em}\oldtocsubsection{#1}{#2}}
\let\oldtocsubsubsection=\tocsubsubsection
\renewcommand{\tocsubsubsection}[2]{\hspace{4.2em}\oldtocsubsubsection{#1}{#2}}

\begin{document}

\title[Robust  unfoldings]{~\vspace{0.25cm}Robust degenerate unfoldings of cycles and tangencies}
\author[Barrientos]{Pablo G. Barrientos}
\address{\centerline{Instituto de Matem\'atica e Estat\'istica, UFF}
    \centerline{Rua M\'ario Santos Braga s/n - Campus Valonguinhos, Niter\'oi,  Brazil}}
\email{pgbarrientos@id.uff.br}
\author[Raibekas]{Artem Raibekas~\vspace{0.30cm}}
\address{\centerline{Instituto de Matem\'atica e Estat\'istica, UFF}
   \centerline{Rua M\'ario Santos Braga s/n - Campus Valonguinhos, Niter\'oi,  Brazil}}
\email{artem@mat.uff.br}


\begin{abstract}
 We construct open sets of {degenerate} unfoldings
of heterodimensional cycles of any co-index $c>0$ and homoclinic
tangencies of arbitrary codimension $c>0$.
These {type of} sets are known to be the support
of unexpected phenomena in families of diffeomorphisms, such as
the Kolmogorov typical co-existence of infinitely many attractors.
{As a prerequisite}, we also construct robust homoclinic
tangencies of large codimension
 {which cannot be inside a strong partially
hyperbolic set}.
\end{abstract}
\setcounter{tocdepth}{1} 
 \maketitle
\section{Introduction}
Robust homoclinic tangencies and robust heterodimensional cycles
are, in general, pre\-requisites for obtaining abundant
complicated dynamical systems~\cite{New70,GS72,New74,BD08,BD12}.
Both configurations imply the existence of a non-transversal
intersection between the stable and unstable manifolds of points
in the same or in different transitive hyperbolic sets. A priori,
the non-transverse intersection could be destroyed by a small
perturbation. But since it is robust, this means that a new
non-transverse intersection is created between the manifolds of
the continuation of the hyperbolic sets. The unfolding of these
bifurcations yields a great number of changes in the dynamics. For
instance, infinitely many saddle periodic points and sinks appear
in the unfolding of homoclinic tangencies. Hence, the persistence
of these bifurcations allowed to get a generic coexistence of
infinitely many periodic attractors~\cite{New79,GTS93,PV94,GST08}.

The construction of robust tangencies in lower dimension is based
on the creation of thick horseshoes involving distortion estimates
which are typically $C^2$. However, in higher dimensions it was
possible to construct robust homoclinic tangencies in the
$C^1$-topology using blenders~\cite{Ao08,BD12}. Blenders are
hyperbolic sets having a thicker invariant manifold than initially
expected. They were discovered by Bonatti and Diaz~\cite{BD96} and
now are essential objects in the study of non-hyperbolic dynamics.
On the other hand, all of the above mentioned constructions are of
codimension one. That is, the dimension of the coincidence of the
tangent spaces at the tangent point. Recently in~\cite{BR17}, the
authors gave the first examples of $C^2$-robust tangencies of
large codimension. The novelty in the construction was the use of
the blender for the dynamics induced in the tangent bundle.

A different approach, when compared to the generic results
mentioned above, is to look for bifurcations of homoclinic
tangencies in \textit{parametric families} of diffeomorphisms. For
decades it was thought that the coexistence of infinitely many
hyperbolic attractors was meager in families of dynamical
systems~\cite{PS95}. However, recently and far from intuition,
Berger showed in~\cite{Ber16} that actually these phenomena form a
residual set.
Behind this result was the construction of open sets of families
of endomorphisms with robust homoclinic tangencies and with an
extra property: the family unfolds  {degenerately} a tangency.
This means the unfolding is slow in the sense of the zeroing of
the first terms in a certain Taylor polynomial describing the
local separation between the manifolds. Although a {degenerate}
unfolding of a tangency could be destroyed by a small perturbation
of the family, this perturbation has another tangency which
unfolds  {also degenerately}. The mechanism involved in these
constructions of Berger~\cite{Ber16} is also the blender, but now
constructed for the dynamics induced in the space of jets (the
space of velocities). See also \cite{BCP17,Ber17}.

The objective of the present work is to unite the construction
of~\cite{BR17} and \cite{Ber16} to obtain robust {degenerate}
unfoldings of homoclinic tangencies of large codimension for
families of diffeomorphisms. Only this will not be done by merely
combining the two previous results. Here we present a new method
of construction of robust tangencies of large codimension,
different from the one in~\cite{BR17}. This is a generalization to
higher codimension of the construction in~\cite{BD12} using
folding manifolds.
It is expected that the unfolding of these robust degenerated
tangencies gives new dynamical consequences. For example, the
existence of residual sets with infinitely many attracting
invariant tori of large dimension.

%

%

\subsection{Dimension and codimension of an intersection}
{Let us begin with some definitions on the intersections of
submanifolds.} Let $\mathcal{L}$ and $\mathcal{S}$ be submanifolds
of $\mathcal{M}$.  {We say that $\mathcal{L}$ and $\mathcal{S}$
have an intersection of \emph{dimension} $d\geq 0$ at $x\in
\mathcal{L}\cap \mathcal{S}$, if
$$
     d=d_x(\mathcal{L},\mathcal{S})\eqdef \dim T_x\mathcal{L}\cap
     T_x\mathcal{S}.
$$
Notice that $d$ is the maximum number of common linearly
independent tangent directions in $T_x\mathcal{M}$. However, this
number is not enough to measure how far the intersection is from
being transverse. In order to quantify this
we say that
$\mathcal{L}$ and $\mathcal{S}$ have an intersection of
\emph{codimension} $c\geq 0$ if
\begin{equation*}
\label{codimension}
   c=c_x(\mathcal{L},\mathcal{S})\eqdef \dim \mathcal{M}- \dim
   \left(T_x\mathcal{L}+T_x\mathcal{S}\right).
\end{equation*}
An intersection of codimension zero is called \emph{transverse}
and is said to be \emph{tangencial} otherwise. Observe that the definition of
a tangency ($c>0$) includes the case $d=0$ in which
$T_x\mathcal{L}+T_x\mathcal{S}=T_x\mathcal{L}\oplus T_x
\mathcal{S}$. In the literature this case is called as a
\emph{quasi-transverse} intersection. On the other hand, the sum
of the dimension and the codimension of an intersection between
submanifolds is in general not equal to the dimension of
$\mathcal{M}$. Moreover, when the codimension of $\mathcal{L}$
coincides with the dimension of $\mathcal{S}$, then if $x\in
\mathcal{L}\cap \mathcal{S}$
$$
   c_x(\mathcal{L},\mathcal{S})= \dim \mathcal{M}- (\dim T_x\mathcal{L} +
   \dim T_x \mathcal{S} - \dim T_x\mathcal{L} \cap T_x
   \mathcal{S}) = d_x(\mathcal{L},\mathcal{S}).
$$
}

 \begin{figure}
 \footnotesize
  \begin{center}
   \begin{picture}(400,165)
\put(10,0){\subfigure[]{\label{fig1a}
\includegraphics[scale=0.6]{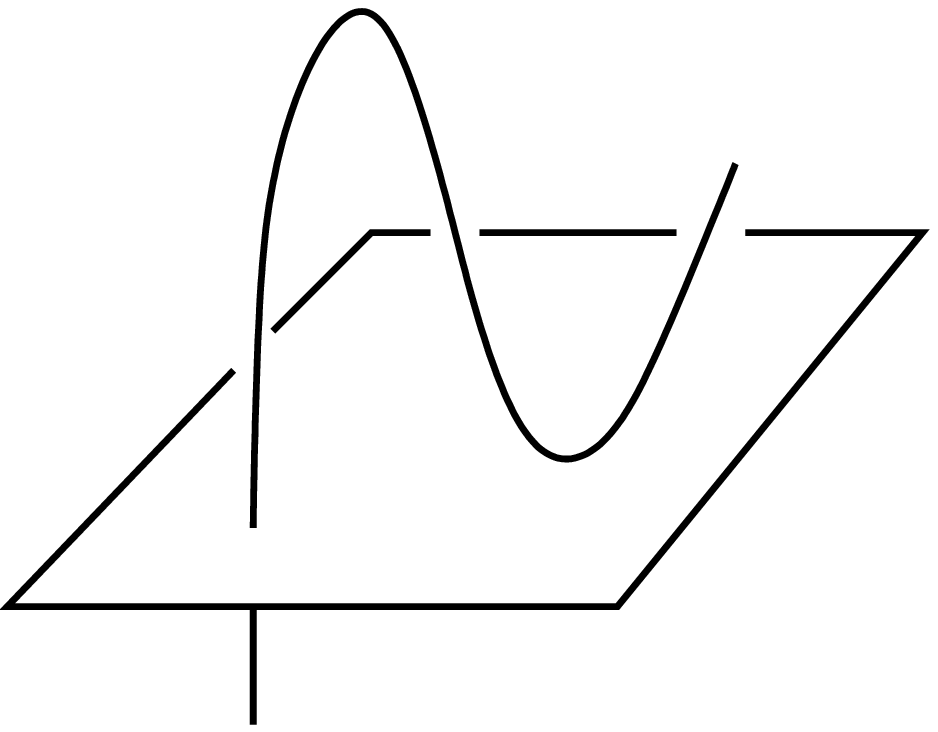}}}
\put(230,0){\subfigure[]{\label{fig1b}\includegraphics[scale=0.4]{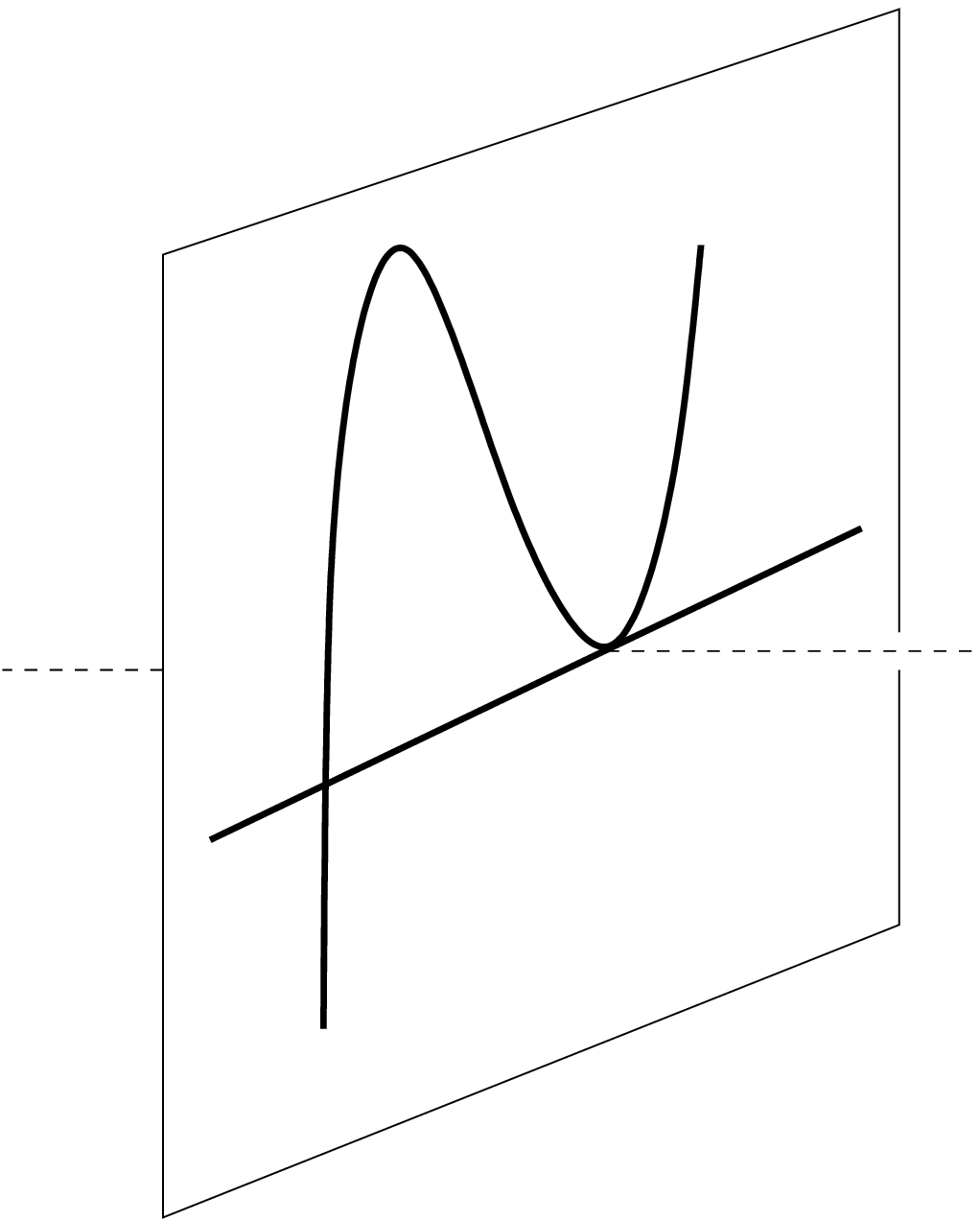}}}
\put(60,40){\normalsize$x_1$} \put(110,38){\normalsize $x_2$}
\put(145,100){\normalsize $\mathcal{L}$} \put(153,75){\normalsize
$\mathcal{S}$} \put(273,45){\normalsize$y_1$}
\put(273,45){\normalsize$y_1$} \put(300,59){\normalsize$y_2$}
\put(315,120){\normalsize $\mathcal{L}$} \put(330,86){\normalsize
$\mathcal{S}$}
\end{picture}
~\\[1.2cm]
 \caption{
 Intersection of two submanifolds in a three dimensional space.
 Figure~\ref{fig1a} shows a transverse intersection at the
 point $x_1$ and a tangency of
 codimension $c=1$ at $x_2$ (and dimension also $d=1$).
 Figure~\ref{fig1b} shows
 intersections of codimension $c=1$ and $c=2$ at $y_1$ and $y_2$
 respectively. At $y_1$ we have a quasi-transverse intersection (dimension $d=0$)
 and at  $y_2$ we have a tangency of dimension $d=1$.
 }\label{fig11}
 \end{center}
 \end{figure}

\subsection{Heterodimensional cycles and homoclinic tangencies}
 A $C^r$-diffeomorphism $f$ of a manifold $\mathcal{M}$ has a \emph{homoclinic tangency of codimension $c>0$}
 if there is a pair of points $P$ and $Q$, in the same transitive
hyperbolic set, so that the unstable invariant manifold of $P$ and
the stable invariant manifold of $Q$ have  {an intersection of
codimension~$c$} at a point $Y$. That~is,
$$
  Y \in W^u(P)\cap W^s(Q) \quad \text{and} \quad  c= c_Y( W^u(P),W^s(Q)).
$$
 {Observe that actually, as the codimension of
$W^u(P)$ coincides with the dimension of $W^s(Q)$ we have that also in
this case
$$
   c=d_Y(W^u(P),W^s(Q)).
$$}
Similarly,  $f$  has a \emph{heterodimensional cycle of co-index
$c>0$} if there exist two transitive hyperbolic sets $\Lambda$ and
$\Gamma$ such that their invariant manifolds meet cyclically and
$|\mathrm{ind}^s(\Lambda)-\mathrm{ind}^s(\Gamma)|=c$. Here
$\mathrm{ind}^s(\cdot)$ denotes the dimension of the stable bundle
of the  respective set.
 {By means of an arbitrarily small perturbation if
necessary,  the stable and unstable manifolds have a
transverse intersection of dimension $c$ and a tangency  which is
a quasi-transverse intersection of codimension $c$. Indeed,
we can assume that
$\mathrm{ind}^s(\Lambda)-\mathrm{ind}^s(\Gamma)=c$
and for $P\in
\Lambda$, $Q\in \Gamma$ suppose that $Y$ belongs to
$W^s(P)\cap W^u(Q)$. Then $\dim T_YW^s(P)+\dim T_YW^u(Q)=\dim
\mathcal{M}+c$ and thus, in general,
$$
   d_Y(W^s(P),W^u(Q))=c \quad \text{and} \quad c_Y(W^s(P),W^u(Q))=0.
$$
On the other hand, if $Y\in W^u(P)\cap W^s(Q)$ then $\dim
T_YW^u(P)+\dim T_YW^s(Q)=\dim \mathcal{M}-c$ and hence,
$$
   d_Y(W^u(P),W^s(Q))=0 \quad \text{and} \quad
   c_Y(W^u(P),W^s(Q))=c.
$$
}

\subsection{Robust tangencies of large codimension}
{By Kupka-Smale's theorem, $C^r$-generically, the stable and
unstable manifolds of a pair of saddle hyperbolic periodic points
meet transversally. Hence, tangencies associated with saddles
occurs in the complement of a residual set of diffeomorphisms and
thus are non-generic dynamical configurations. However, the
situation is different if instead of the periodic saddles we
consider non-trivial hyperbolic sets.
It is well-known the existence of open sets of diffeomorphisms
displaying non-transverse intersections between the stable and
unstable manifolds of points in the continuations of these
hyperbolic sets.
That is, the so-called $C^r$-open sets of diffeomorphisms with
robust tangencies or robust heterodimensional cycles.
See~\cite{New70,GTS93,PV94,BD12} for robust homoclinic tangencies
of codimension one and \cite{BD08} and reference therein for
robust heterodimensional cycles.}

Robust homoclinic tangencies of large codimension in the
$C^2$-topology were recently discovered in~\cite{BR17} inside
\emph{strong partially hyperbolic sets}. That is, invariant sets
with a dominated splitting of the form $E^s\oplus E^c \oplus E^u$,
where $E^s$ and $E^u$  are the non-trivial contracting and
expanding bundles respectively. Here, we will construct new
examples
of a different nature from the robust tangencies of large
codimension showed in~\cite{BR17}. This is because they cannot be
embedded inside a strong partially hyperbolic set. At the point of
tangency, 
the splitting is of the form $E^s\oplus E^c$, where $E^c$ cannot
be divided into neither contracting nor expanding subbundles. In
this case, we say that the tangency is inside a \emph{weak
partially hyperbolic set}.

\begin{mainthm}
\label{thmD} Every manifold of dimension $m> c^2+c$ admits a
diffeomorphism having a $C^{2}$-robust homoclinic tangency of
codimension $c>0$ inside a weak partially hyperbolic set.
\end{mainthm}


Notice that the above theorem gives as a particular case the
well-known results~\cite{GTS93,PV94,Ro95} about $C^2$-robust
homoclinic tangencies of codimension one in higher dimensions.
Here, we provide a different proof inspired by the construction of
$C^1$-robust homoclinic tangencies of Bonatti and D\'iaz
in~\cite{BD12}. The concepts of \emph{folding manifolds} and
\emph{blenders} constructed in the \emph{tangent bundle} allows us
to extend their result to large codimension.

\subsection{{Degenerate} unfoldings}
A  {tangency} at a point $Y$ between the unstable manifold
$W^u(P)$ and the stable manifold $W^s(Q)$ of a
$C^r$-diffeomorphism $f$ can be unfolded by considering
$C^d$-families $(f_a)_a$ of $C^r$-diffeomorphisms
pa\-ra\-me\-te\-ri\-zed by $a\in \mathbb{I}^k$ with $f_0=f$ and
$\mathbb{I}=[-1,1]$. {We will suppose $0< d\le r<\infty$ and
$k\geq 1$.}
Many articles usually impose a generic condition on the velocity
of the unfolding. {They assume} that the distance between the
manifolds has positive derivative with respect to the parameter:
$$
   \frac{d\delta}{da}(0)\not=0 \quad \text{where} \ \
   \delta(a)=\min \{d(x,y): x\in W^u(P_a)\cap U, \ \  y\in W^s(Q_a)\cap U\}.
$$
Here $P_a$, $Q_a$ are the continuations of the hyperbolic saddles
$P_0=P$ and $Q_0=Q$ for $f_a$  respectively and $U$ is a small
neighborhood of $Y$. However, in this work we are interested in
studying {unfoldings} where this generic assumption fails.

{Let us consider first the case that the family $f=(f_a)_a$
unfolds a heterodimensional cycle of co-index $c>0$ at $a=0$. That
is,
the above points $P_0$ and $Q_0$  now belong, respectively, to
transitive hyperbolic sets $\Lambda_0$ and $\Gamma_0$ of $f_0$
with co-index $c$ and whose stable and unstable manifolds
intersect cyclically. Moreover, say that $Y\in W^u(P_0)\cap
W^s(Q_0)$ is the tangency of the cycle (i.e, the intersection that
in general could be assumed quasi-transverse and of codimension
$c$). The unfolding of this heterodimensional cycle
 is said to be \emph{$C^d$-{degenerate}} at $a=0$
if there exist
$$
 p_a \in W^u_{loc}(P_a) \ \ \text{and} \ \ q_a\in  W^s(Q_a) \ \
\text{so that} \ \ d(p_a,q_a)=o(\|a\|^{d})\ \ \text{at $a=0$}
$$
where 
$p_0=q_0=Y$ and $p_a$, $q_a$ vary $C^d$-continuously with respect
to the parameter $a\in \mathbb{I}^k$.}

 {Now we consider that the family $f=(f_a)_a$
unfolds  a homoclinic tangency of codimension $c>0$ at $a=0$. That
is, we assume
the points $P_0$ and $Q_0$ belong to the same hyperbolic set
$\Lambda_0$ of $f_0$ and that the homoclininc tangency $Y \in
W^u(P_0)\cap W^s(Q_0)$ has {codimension} $c>0$. The unfolding of
this homoclinic tangency  of codimension $c>0$ is said to be
\emph{$C^d$-{degenerate}} at $a=0$ if there are points $p_a \in
W^u_{loc}(P_a)$, $q_a \in W^s(Q_a)$ and $c$-dimensional subspaces
$E_a$ and $F_a$ of $T_{p_a}W^u(P_a)$ and $T_{q_a}W^s(Q_a)$
respectively such that
$$
 d(p_a,q_a)=o(\|a\|^{d}) \quad \text{and} \quad
d(E_a,F_a)={o(\|a\|^{d})} \quad \text{at $a=0$}.
$$
Here,} $p_0=q_0=Y$ and $(p_a,E_a)$, $(q_a,F_a)$ vary
$C^d$-continuously with respect to the parameter $a\in
\mathbb{I}^k$. Observe that in this case it is necessary to assume
that $d<r$ because the above definition involves the dynamics of
the family $(f_a)_a$ in the tangent bundle (in fact, in certain
Grassmannian bundles). In~\cite{Ber16} $C^d$-degenerate unfoldings
of homoclinic tangencies were called for short
\emph{$C^d$-paratangencies}. {The key consequence of having a
$C^d$-paratangency at $a=0$ is that one can perturb the family and
obtain a new family which now has a \emph{persistent homoclinic
tangency} in the sense of~\cite{Ber17}. That is, a tangency point
$Y_a$ between the stable and unstable manifold which varies
$C^d$-continuously with respect to $a$ in an open set of
parameters $J\subset \mathbb{I}^k$ containing $a=0$.
}

\subsection{Open sets of families with {degenerate} unfoldings}


{In order to be more precise, we now introduce the following
definitions. The exact notion of a $C^d$-family of
$C^r$-diffeomorphisms and the $C^{d,r}$-topology considered is
going to be defined in \S\ref{sec:topology}. }

A $k$-parameter $C^{d}$-family $f=(f_a)_a$ of
$C^r$-diffeomorphisms $f_a$   {displays} a 
\begin{enumerate}[leftmargin=0.9cm,itemsep=0.2cm]
\item[-] \emph{$C^{d,r}$-robust $C^d$-{degenerate} unfolding of
a homoclinic tangency of codimension $c$} at $a=0$ if there are a
transitive hyperbolic set $\Lambda_0$ of $f_0$ and a
$C^{d,r}$-neighborhood $\mathscr{U}$ of $f$, such that any
$g=(g_a)_a \in \mathscr{U}$ {displays} a $C^d$-{degenerate}
unfolding of a homoclinic tangency of codimension $c>0$ at $a=0$
associated with the continuations of $\Lambda_{0}$ for $g_0$.
\item[-] \emph{$C^{d,r}$-robust $C^d$-{degenerate} unfolding of a
heterodimensional cycle of co-index $c$} at $a=0$ if there are
transitive hyperbolic sets $\Lambda_0$ and $\Gamma_0$ of $f_0$
with co-index $c$ and a $C^{d,r}$-neighborhood $\mathscr{U}$ of
$f$, such that any $g=(g_a)_a \in \mathscr{U}$ {displays} a
$C^d$-{degenerate} unfolding of a heterodimensional cycle of
co-index $c>0$ at $a=0$ associated with the continuations of
$\Lambda_{0}$ and $\Gamma_0$ for $g_0$.
\end{enumerate}

For simplicity, we have chosen $a=0$ as the critical parameter of
the unfolding. However, {degenerate} unfoldings can also be
introduced at any other parameter $a=a_0$ with $a_0\in
\mathbb{I}^k$. We say that $f=(f_a)_a$ {displays} a
$C^{d,r}$-robust $C^d$-{degenerate} unfolding of a
heterodimensional cycle (or a tangency) at \emph{any parameter}
when any $g=(g_a)_a \in \mathscr{U}$ has a heterodimensional cycle
(or a  {homoclinic} tangency) at $a=a_0$ which unfolds
$C^d$-{degenerate} for all $a_0\in \mathbb{I}^k$. Robust
{degenerate} unfoldings at any parameter are involved in
unexpected phenomena as the typical coexistence of infinitely many
sinks~\cite{Ber16,Ber17}, infinitely many non-hyperbolic strange
attractors~\cite{R17} and fast growth of periodic
points~\cite{Ber19} among others.





\begin{mainthm}
\label{thmC} Any manifold of dimension $m> c^2+c$ admits a
$k$-parameter $C^d$-family of $C^r$-diffeomor\-phisms with
{$0<d<r-1$}, which  {displays}  a $C^{d,r}$-robust
$C^d$-{degenerate} unfolding of a homoclinic tangency of
codimension $c>0$ at any parameter.
\end{mainthm}

%

 We will also show the existence of robust {degenerate}
unfoldings of heterodimensional cycles of any co-index at any
parameter:

\begin{mainthm}
\label{thmA} Any manifold of dimension $m> 1+c$ admits a
$k$-parameter $C^d$-family of $C^r$-diffeomor\-phisms with {$0<d<
r$} which  {displays} a $C^{d,r}$-robust $C^d$-{degenerate}
unfolding of a heterodimensional cycle of co-index $c>0$ at any
parameter.
\end{mainthm}

{The differences in the regularity and dimension that appear in
the above theorems come from the  nature of the unfolding of the
 {tangency}, as we now explain. With respect to the
regularity assumption in Theorem~\ref{thmA}, the unfolding of a
heterodimensional cycle only deals with the distance in the
ambient manifold. There is a loss of a derivative ($d<r$) in the
moment that we pass to study the kinematic of the movement by
lifting the family to the space of velocities (jet space). This is
because we will need that the induced dynamics in the jet space is
a $C^1$-diffeomorphism. On the other hand, the unfolding of
tangencies in Theorem~\ref{thmC} requires first to lift the
dynamics to the space where the bifurcation is produced, that is
to the Grassmannian bundle.
After that one needs to perform a similar analysis in the space of
velocities to study the corresponding $C^d$-{degenerate}
unfolding. This provides a loss of two degrees of regularity
($d<r-1$) as is claimed in Theorem~\ref{thmC}. Moreover, we will
need to use Theorem~\ref{thmD} to obtain Theorem~\ref{thmC}, thus
explaining the dimension of the manifold, $m> c^2+c$, as is
related to the codimension of the tangency.} 

\subsection{Topology of families of diffeomorphisms}
\label{sec:topology} Set $\mathbb{I}=[-1,1]$. Given $0<d \leq
r\leq \infty$, $k\geq 1$ and a manifold $\mathcal{M}$, we denote
by $C^{d,r}(\mathbb{I}^k,\mathcal{M})$ the space of $C^d$-families
$f=(f_a)_a$ of $C^r$-diffeomorphisms $f_a$ of $\mathcal{M}$
parameterized by $a\in\mathbb{I}^k$ such that
$$
  \partial^i_a \partial^j_x f_a(x) \ \ \text{exists continuously
  for all $0\leq i\leq d$, \ \   $0\leq i+j\leq r$  \ \  and \ \
   $(a,x)\in \mathbb{I}^k\times \mathcal{M}.$}
$$
 {We endow this space with the topology given by the
$C^{d,r}$-norm
$$
\|f\|_{{C}^{d,r}}=\max\{\sup \|\partial^i_a\partial_x^j f_a (x):
\, 0\leq i \leq d, \  0\leq i + j \leq r\} \quad \text{where \
$f=(f_a)_a \in {C}^{d,r}(\mathbb{I}^k,\mathcal{M})$.}
$$
}%
In what follows we restrict our attention to $C^{d}$-families
$f=(f_a)_a$ of $C^r$-diffeomorphisms $f_a$ of a manifold
$\mathcal{M}$ of dimension $m\geq 3$.


\subsection{Structure of the paper}
Section \S\ref{sec:blender} contains the definition of a
\emph{blender}, one of the main tools in this paper.
In section~\S\ref{sec:tangencias} we prove Theorem~\ref{thmD}.
After that, we describe formally the notion of {degenerate}
unfoldings in section~\S\ref{sec:unfoding}. In
\S\ref{sec:parablender} we recall and develop the notion of
\emph{parablenders}, the second main tool of the paper. Finally in
sections \S\ref{sec:cycles} and \S\ref{sec:final} we prove
{Theorems}~\ref{thmA} and~\ref{thmC} respectively.

\section{Blenders}
\label{sec:blender} We attribute the following definition to
Bonatti and D\'iaz (see~\cite{BBD16}).  Blenders were initially
defined having  {central dimension $c=1$}
(see~\cite{BD96,BDV05,BD12}) and blenders with large
 {central dimension} were first studied
in~\cite{NP12,BKR14,BR17}.  {After that, they  also appeared
in~\cite{BR18,ACW17} and in holomorphic dynamics
in~\cite{biebler2016persistent,dujardin2017non,taflin2017blenders}.}

\begin{defi} \label{def:blender} Let $f$ be a  $C^r$-diffeomorphism of a manifold $\mathcal{M}$.
A {non-empty} compact set $\Gamma \subset \mathcal{M}$ is a
\emph{$cs$-blender} of  {central dimension $c\geq 1$} if
\begin{enumerate}
\item $\Gamma$ is a transitive, {maximal invariant hyperbolic set
in the closure of a neighborhood $\mathcal U$ having a partially
hyperbolic splitting}
$$
    \text{$T_\Gamma \mathcal{M}= 
    E^{ss} \oplus E^c \oplus E^{u}$}
$$
where $E^s=E^{ss}\oplus E^c$ is the stable bundle,
${d_{ss}}=\dim E^{ss}\geq 1$ and $c=\dim E^{c}\geq 1$, 
\item  there exists a {non-empty} open set $\mathscr{D}$ of $C^1$-embeddings of
$d_{ss}$-dimensional discs into $\mathcal{M}$, and
\item there exists a $C^1$-neighborhood $\mathscr{U}$ of $f$,
\end{enumerate}
such that
$$
  W^u_{loc}(\Gamma_g) \cap \mathcal{D} \not = \emptyset \quad \text{for all
  $\mathcal{D}\in \mathscr{D}$ and $g\in \mathscr{U}$}
$$
where $\Gamma_g$ is the continuation of $\Gamma$ for $g$ and
{$W^u_{loc}(\Gamma_g)=\{x\in \mathcal{U}: g^{-n}(x)\in \mathcal{U}
\ \text{for all $n\geq 0$}\}$}. The set $\mathscr{D}$ is called
{a} \emph{superposition region} of the blender. Finally, a
\emph{$cu$-blender of  {central dimension} $c$} is $cs$-blender of
 {central dimension} $c$ for $f^{-1}$.
\end{defi}

The hyperbolicity of $\Gamma$ implies that given a point $x\in
W^u_{loc}(\Gamma) \cap \mathcal{D}$, there is a point $z\in
\Gamma$ such that $x \in W^u_{loc}(z)\cap \mathcal{D}$. Observe
that the local unstable manifold of $z$ is a $C^1$-embedded disc
of dimension  {$d_u=\dim E^u$} and $\mathcal{D}$ is a
 {$d_{ss}$}-dimensional disc. These two discs  are
in {\emph{relative general position}} if it holds that
$$
   T_x W^u_{loc}(z) + T_x \mathcal{D} = T_x W^u_{loc}(z)\oplus
   T_x\mathcal{D}.
$$
In this case,
we have an {intersection} of codimension
$$
  c_x(W^u_{loc}(z),\mathcal{D})=\dim \mathcal{M} - \dim(T_x W^u_{loc}(z)+ T_x\mathcal{D}) =
  \dim \mathcal{M} -  {(d_u+ d_{ss})}=c \geq 1.
$$
 {Thus, $W^u_{loc}(z)$ and $\mathcal{D}$
have a tangency of codimension at least $c$, which is in general,
a quasi-transverse intersection of codimension exactly $c$.}


\subsection{Covering criterium} \label{sec:covering-criterium}
In \cite{BKR14,BR17} blenders of large  {central dimension} were
constructed by using the covering criterium. Namely, we {consider}
$C^1$-diffeomorphisms 
which are locally defined as a skew-product as explained below.

First, we consider a $C^1$-diffeomorphism $F$ of a manifold
$\mathsf{N}$ having a horseshoe $\mathsf{\Lambda}$ {contained in a
local chart} which is the maximal $F$-invariant  set in the
closure of some bounded open set $\mathsf{R}$ of $\mathsf{N}$. The
horseshoe has stable index (dimension of the stable bundle) equal
to $d_{ss}=\mathrm{ind}^s(\mathsf{\Lambda})>0$  and satisfies that
\begin{enumerate}
\item $F|_{\Lambda}$ is conjugate to a shift of $\kappa$-symbols
and
\item there exists $0<\nu<1$ such that
\begin{equation} \label{eq:nu}
     m(DF(x)) \leq \nu < 1<\nu^{-1} \leq \|DF(x)\| \qquad \text{for all $x\in
     \mathsf{\Lambda}$.}
\end{equation}
\end{enumerate}
Here $m(T)=\|T^{-1}\|^{-1}$ denotes the co-norm of a linear
operator $T$. Let $\{\mathsf{R}_1,\dots,\mathsf{R}_\kappa\}$ be an
open covering of $\Lambda$, whose intersection with $\Lambda$ is a
Markov partition. There is no loss of generality in assuming that
$\mathsf{R}=\mathsf{R}_1\cup\dots\cup \mathsf{R}_\kappa$
 {with} $\mathsf{R}_\ell=(-2,2)^{d_{ss}}\times
I_\ell$, where $I_\ell$ is a product of  {$d_u$} open intervals in
$[-2,2]$ with $\dim \mathsf{N}= d_{ss}+  {d_{u}}$ for
$\ell=1,\dots,\kappa$.  {Moreover, from the hyperbolicity of
$\Lambda$, we can assume that there is a $DF^{-1}$-invariant
cone-field on $\mathsf{R}$:
\begin{enumerate}[start=3]
\item there exist $\alpha>0$ such that
$$DF^{-1}(x)\mathcal{C}^{ss}_\alpha \subset
\mathcal{C}^{ss}_{\nu^2\alpha} \quad \text{for all
   $x\in \mathsf{R}$.}
$$
\end{enumerate}
Here, for a given $\theta>0$, we denote
\begin{equation} \label{strong-stable-cone}
  \mathcal{C}^{ss}_\theta\eqdef 
  \{(u,v)\in
  \mathbb{R}^{d_{ss}}\oplus
  \mathbb{R}^{d_u}: \, \|v\|<\theta\|u\|
  \} \cup \left\{0\right\}.
\end{equation}
We will call $\mathcal{C}^{ss}_\alpha$ as a \emph{stable cone-field
on $\mathsf{R}$ of $F$} and refer to the parameter
$\alpha$ as the \emph{width} of the cone. }

Now take $C^1$-diffeomorphisms $\phi_1,\dots,\phi_\kappa$ of
another manifold $M$ of dimension $c>0$, which are local
$(\lambda,\beta)$-contractions in a bounded open set $D \subset
M$, with $0<\lambda<\beta<1$:
$$
  \phi_\ell(\overline{D})\subset D \quad \text{and} \quad
\lambda<m(D\phi_\ell(y))<\|D\phi_\ell(y)\|<\beta<1 \quad \text{for
all $y\in \overline{D}$ \  and \ $\ell=1,\dots,\kappa$.}
$$
Finally, we consider a $C^1$-diffeomorphism $\Phi$ of
$\mathcal{M}=\mathsf{N}\times M$ locally defined as a skew-product
$$
  \Phi=F\ltimes (\phi_1,\dots,\phi_\kappa) \quad \text{on} \ \
  \mathcal{U}=(\mathsf{R}_1 \times D) \cup \dots \cup (\mathsf{R}_\kappa \times D)
$$
so that
$$
 \Phi(x,y)=(F(x),\phi(x,y)) \quad \text{with \ \ $\phi(x,y)=\phi_\ell(y)$ \ \ if \ \ $(x,y)\in \mathsf{R}_\ell\times
 D$.}
$$

\begin{notation}
In  the rest of the paper, we will use the notation
$$
\Psi=G\ltimes(\psi_1,\dots,\psi_\kappa) \quad \text{on} \ \
\mathcal{V}=\mathcal{V}_1\cup \dots \cup \mathcal{V}_\kappa
$$
to define the skew-product map $\Psi(x,y)=(G(x),\psi_\ell(y))$
with $(x,y)\in \mathcal{V}_\ell$ for $\ell=1,\dots,\kappa$, {where
$\mathcal{V}_1, \dots, \mathcal{V}_\kappa$ are pairwise disjoint
sets.}
\end{notation}

The following theorem {from~\cite[Thm.~C]{BKR14} and
\cite[Thm.~3.8]{BR17}} shows that under the assumption of
domination and the covering criterium, the map $\Phi$ has a
$cs$-blender of  {central dimension $c\geq 1$}.

\begin{thm}
\label{thmBKR} Let $\Phi$ be a $C^1$-diffeomorphism of a manifold
$\mathcal{M}$ locally defined as a skew-product $\Phi=F\ltimes
(\phi_1,\dots,\phi_\kappa)$ on $\mathcal{U}$ as above. Assume that
\begin{enumerate}[itemsep=0.05cm]
\item \label{item:domination} the hyperbolic base $F|_{\mathsf{\Lambda}}$ dominates the
fiber dynamics $\phi_\ell$, i.e, it holds that $\nu<\lambda$,
\item \label{item:cover} there exists an open
set $B \subset D$ such that $
   \overline{B} \subset  \phi_1(B) \cup \dots \cup \phi_\kappa(B).
$
\end{enumerate}
Then the maximal invariant set $\Gamma$ of $\Phi$ in
$\overline{\mathcal{U}}$ is a $cs$-blender of
 {central dimension} $c$.  {The superposition
region of the blender is the family of $( {\alpha,}
\nu,\delta)$-horizontal, $d_{ss}$-dimensional $C^1$-discs into
$\mathcal{B}=\mathsf{R}\times B$, where $0<\delta<\lambda L/2$
and $L>0$ is the Lebesgue number of the cover of $B$
in~\eqref{item:cover}. }
\end{thm}

The open set $\mathcal{B}=\mathsf{R}\times B$ of $\mathcal{M}$ is
called {a \emph{superposition domain}. Also, in the above theorem
appears the notion of a family of $( {\alpha,}
\nu,\delta)$-horizontal discs in $\mathcal{B}$ that we define as
follows.}

{
 A proper $C^r$-embedded $d_{ss}$-dimensional
disc $\mathcal{D}$ into $\mathcal{B}=\mathsf{R}\times B$ (or a
\emph{$d_{ss}$-dimensional $C^r$-disc in $\mathcal{B}$} for short)
will be an injective $C^r$-immersion $\mathcal{D}: [-2,2]^{d_{ss}}
\to \overline{\mathcal{B}}$ of the form,
$$
   \mathcal{D}(\xi)=(\xi,g(\xi),h(\xi))\in [-2,2]^{d_{ss}}\times I_\ell \times
   B \quad \text{for} \ \ \xi \in [-2,2]^{d_{ss}} \ \text{and some} \ \ell\in \{1,\dots,\kappa\}.
$$
As usual, we will identify the embedding $\mathcal{D}$ with its
image $\mathcal{D}([-2,2]^{d_{ss}})$.} 

\begin{defi} \label{def:almost-horizontal-disc}
We say that a $d_{ss}$-dimensional $C^1$-disc $\mathcal{D}$ in
$\mathcal{B}=\mathsf{R}\times B$  is \emph{{$( {\alpha,}
\nu,\delta)$}-horizontal} if
\begin{enumerate}[itemsep=0.1cm]
\item \label{disc1}  {$\|Dg\|_\infty\leq \alpha$,}
\item \label{disc2} there is a point $y \in B$ such that
$d(y,h(\xi))<\delta$ for all $\xi\in [-2,2]^{d_{ss}}$,
\item \label{disc3} ${C\cdot\nu} < \delta$ where $C\geq 0$ is a Lipschitz constant
of $h$, i.e.,
$$d(h(\xi),h(\xi'))\leq C\, d(\xi,\xi') \quad \text{for all
$\xi,\xi'\in [-2,2]^{d_{ss}}$.}
$$
\end{enumerate}
\end{defi}

Since $\mathcal{D}$ is a $C^1$-disc notice that $C$ is any
positive constant satisfying $\|Dh\|_{\infty} \leq C$. If {$C=0$}
we say that $\mathcal{D}$ is \emph{horizontal}.
 {Notice that by condition~\eqref{disc1}, the disc
$\mathcal{D}$ is tangent to the stable cone-field defined on
$\mathcal{B}$. Moreover, from condition~\eqref{disc2},
$\mathcal{D}$ is $C^0$-close to an horizontal disc. On the other
hand, although $\mathcal{D}$ may not be $C^1$-close to a
horizontal disc, condition~\eqref{disc3} asks that we still have a
good control of the distortion. Finally,} for a fixed
 {$\alpha$,} $\nu$ and $\delta$ under
 {the conditions in Theorem~\ref{thmBKR}}, the set
of $( {\alpha,}\nu,\delta)$-horizontal discs in $\mathcal{B}$ is
said to be, for short, the family of \emph{almost-horizontal
discs}.

In the next sections we  {construct} diffeomorphisms having robust
 {tangencies}  in any manifold of dimension $m\geq
3$.
Our constructions will use the following particular class of
blenders  {obtained} from the covering criterium.

\subsection{Affine blender}
\label{sec:affine-blender} {We will introduce a class of
$C^r$-diffeomorphims $f$ of $\mathbb{R}^m = \mathbb{R}^n \times
\mathbb{R}^c$ with $r\geq 1$, $n=ss+u\geq 2$ and $c\geq 1$. To do
this, consider first} a $C^r$-diffeomorphism $F$ of $\mathbb{R}^n$
having a horseshoe $\mathsf{\Lambda}$ in the the open cube
$\mathsf{V}=(-2,2)^n$. The horseshoe has stable index
$ss=\mathrm{ind}^s(\mathsf{\Lambda})>0$ and {$F|_{\Lambda}$ is
conjugate to a full shift of a large number $\kappa$ of symbols to
be specified later.} We notice that this number will depend only
on the dimension $c$. For simplicity, assume that
$$
\mathsf{R}_\ell= (-2,2)^{ss} \times I_\ell, \quad
\ell=1,\dots,\kappa
$$ is a Markov partition of $\mathsf{\Lambda}$
where $I_\ell$ is an open disc in $[-2,2]^u$  and $F$ is affine on
each rectangle $\mathsf{R}_\ell$.  {More precisely}, there are
$0<\nu<1$ and linear maps $S_\ell: \mathbb{R}^{ss}\to
\mathbb{R}^{ss}$ and $U_\ell: \mathbb{R}^{u} \to \mathbb{R}^u$
such that
$$
 DF=
\begin{pmatrix}
    S_\ell & 0  \\
    0 & U_\ell
\end{pmatrix}
\quad  \text{on \ \ $\mathsf{R}_\ell$ \ \ where \ \ $\|S_\ell \|,\
\|U_\ell^{-1}\| < \nu$ \quad for $\ell=1,\dots,\kappa$.}
$$
 {Notice that,
$$DF^{-1}(x)\mathcal{C}^{ss}_{\alpha}
\subset \mathcal{C}^{ss}_{\nu^2\alpha} \quad \text{for all $x\in
\mathsf{R}=\mathsf{R}_1\cup\dots\cup \mathsf{R}_\kappa$ \ \ and \
\ $\alpha>0$,}
$$
where the cones $\mathcal{C}^{ss}_{\alpha}$ and
$\mathcal{C}^{ss}_{\nu^2\alpha}$ are defined as
in~\eqref{strong-stable-cone}.
In particular, the cone-field $\mathcal{C}^{ss}_\alpha$ is
$DF^{-1}$-invariant.}

Take affine $(\lambda,\beta)$-contractions
$\phi_1,\dots,\phi_\kappa$ on $D=(-2,2)^c$ with
$\nu<\lambda<\beta<1$. That is, $C^r$-diffeo\-morphisms
$\phi_\ell$ of $\mathbb{R}^c$ such that
$\phi_\ell(\overline{D})\subset D$ and there are linear maps
$T_\ell:\mathbb{R}^c\to \mathbb{R}^c$~so~that
$$
\text{$D\phi_\ell(y)=T_\ell$ for all $y\in\overline{D}$ and
$\lambda<m(T_\ell)\leq \|T_\ell\|<\beta$ for
$\ell=1,\dots,\kappa$.}
$$
Moreover, { we ask that there is an open set $B \subset D$}
containing the origin such that
\begin{equation}
\label{eq:cover}
\overline{B}\subset  \phi_1(B) \cup \dots \cup \phi_\kappa(B).
\end{equation}
\begin{exap} \label{exa-afin-blender} 
Take $\phi_{\pm}(t)=\lambda t \pm (1-\lambda)$ for $t\in [-2,2]$
with $1/2<\lambda<1$ and consider
$$
   \phi_\ell = \phi_{\ell_1}\times \dots\times \phi_{\ell_c} \quad \text{on \  $D=(-2,2)^c$ } \quad  \text{for any \ $\ell=(\ell_1,\dots,\ell_c)\in \{-,+\}^c$.}
$$
Observe that here $\kappa=2^c$. It is not difficult to see that
$B=(-1,1)^c$ satisfies~\eqref{eq:cover}.
\end{exap}

Finally we consider a $C^r$-diffeomorphism $\Phi$ of
$\mathbb{R}^m$ locally defined as the skew-product
$$
  \Phi=F\ltimes (\phi_1,\dots,\phi_\kappa) \quad \text{on} \ \
  \mathcal{U}=(\mathsf{R}_1 \times D) \cup \dots \cup (\mathsf{R}_\kappa \times
  D).
$$
According to Theorem~\ref{thmBKR},  the maximal invariant set
$\Gamma$ in $\overline{\mathcal{U}}$ is a $cs$-blender of
 {central dimension}~$c>0$.  {
Moreover, the superposition region is the family of almost-horizontal $d_{ss}$-dimensional
$C^1$-discs in $\mathcal{B}=\mathsf{R}\times B$, where
$\mathsf{R}=\mathsf{R}_1\cup\dots\cup \mathsf{R}_\kappa$.}

\section{Robust homoclinic tangencies}
\label{sec:tangencias} In this section
%
%
we prove Theorem~\ref{thmD}. We provide the existence of
$C^r$-diffeomorphisms with { $r\geq 2$ having $C^2$-robust}
homoclinic tangencies  {of large codimension} by constructing
these objets in local coordinates. Thus, we {may consider}
$\mathbb{R}^m = \mathbb{R}^n \times \mathbb{R}^c$ with $n\geq 2$
and $c \geq 1$. Throughout this section, we ask that $n=ss+u$ and
$c=u^2$ but we keep the notation $u$, $c$ in order to distinguish
coordinates.  {We divide the proof into several parts and for the
convenience of the reader will explain next the ideas involved.}

 {We study the homoclinic tangencies  of
$f$ by analyzing the induced map $f^G$ on Grassmannian manifolds,
and we would like for this induced map to have a blender.
The main idea is to obtain  robust tangencies
for $f$ by means of a robust intersection between the local unstable
manifolds of a blender (for the induced dynamics)
and a
particular disc in the superposition region.
Hence,
in~\S\ref{sec:affine-blenders-grasmannian} we will construct a
class of $C^r$-diffeomorphisms of $\mathbb{R}^m$ which induces a
$cs$-blender $\Gamma^G$ on the Grassmannian manifold. This class
of diffeomorphisms are the locally defined skew-product maps
having an affine blender $\Gamma$ introduced
in~\S\ref{sec:affine-blender} with some additional restrictions.
Afterwards, we introduce in~\S\ref{sec:folding-manifold} the
notion of a folding manifold $\mathcal{S}$ in $\mathbb{R}^m$,
having the
main property
of inducing a
disc $\mathcal{S}^G$ in the superposition region of $\Gamma^G$.
Finally, in \S\ref{sec:robust-tangency-folding-manifold} we show
how the robust intersection between the local unstable manifolds of
$\Gamma^G$ and the induced disc $\mathcal{S}^G$ provides a robust
tangency between the unstable manifolds of $\Gamma$ and the
folding manifold $\mathcal{S}$. One can see the folding manifold as a
piece of a leaf of the stable manifold of $\Gamma$  and then the proof  of
Theorem~\ref{thmD} can be concluded in~\S\ref{sec:prove-ThmD}.}

\subsection{Grassmannian manifold}
\label{sec:grasmanian} Let $f$ be a $C^r$-diffeomorphism of
$\mathbb{R}^m$. We will consider an induced map by $f$ on the
Grassmannian manifold $G_u(\mathbb{R}^m)=\mathbb{R}^m\times
G(u,m)$ given by
$$
  {f}^{{G}} : G_u(\mathbb{R}^m) \to G_u(\mathbb{R}^m), \qquad
  {f}^{{G}}(x,E)=(f(x),Df(x)E)
$$
where $G(u,m)$ is the set of $u$-planes in $\mathbb{R}^m$. Notice
that $f^{{G}}$ is a $C^{r-1}$-diffeomorphism of
${G}_u(\mathbb{R}^m)$.

\subsection{Blender induced on the Grassmannian manifold} \label{sec:affine-blenders-grasmannian}
Fix $r\geq 2$. We will start by considering a
$C^{r}$-diffeomorphism $\Phi$ of $\mathbb{R}^m$ locally defined as
a skew-product $\Phi=F\ltimes (\phi_1,\dots,\phi_\kappa)$ and
having an affine $cs$-blender $\Gamma$, as
in~\S\ref{sec:affine-blender}. Notice that for each
$\ell=1,\dots,\kappa$, the differential map $D\Phi(x,y)$ is the
same linear map $D\Phi_\ell$ for all $(x,y)\in
\mathsf{R}_\ell\times D$. Moreover, $E^u=\{0^{ss}\}\times
\mathbb{R}^u \times \{0^c\}$ is an attracting fixed point of the
action of these maps on $G(u,m)$ with eigenvalues less than
$\beta\nu<1$. Let $\mathcal{C}^{{G}}$ be an open neighborhood of
$E^u$ in $G(u,m)$ so that $D\Phi_\ell \cdot \mathcal{C}^{{G}}
\subset \mathrm{int}(\mathcal{C}^{{G}})$ {for all
$\ell=1,\dots,\kappa$}. The Grassmannian induced map $\Phi^{{G}}$
restricted to $\mathsf{R}_\ell\times D\times \mathcal{C}^{{G}}$ is
given by
$$
\Phi^{{G}}(x,y,E)=(F(x),\phi_i(y),D\Phi_\ell \cdot E) \qquad
\text{for all \ $\ell=1,\dots,\kappa$.}
$$ By a change of coordinates we can write $\Phi^{{G}}$ restricted to
$\mathcal{U}^{{G}}=\mathcal{U}\times \mathcal{C}^{{G}}$ as the
skew-product
$$
    \Phi^{{G}}=F^{{G}} \ltimes (\phi_1,\dots,\phi_\kappa)
    \quad
    \text{on \  \
    ${\mathcal{U}}^{{G}}=(\mathsf{R}^{{G}}_1 \times D)\cup\dots\cup(\mathsf{R}^{{G}}_\kappa\times
    D$)}
$$
where 
$$
F^{{G}}=F \ltimes (D\Phi_1,\dots,D\Phi_\kappa) \quad \text{on \ \
$\mathsf{R}^{{G}}=\mathsf{R}^{{G}}_1\cup\dots\cup\mathsf{R}^{{G}}_\kappa=(\mathsf{R}_1\times\mathcal{C}^{{G}})
\cup \dots \cup (\mathsf{R}_\kappa\times\mathcal{C}^{{G}})$.}
$$
Moreover, $F^{{G}}$ has a horseshoe
${\mathsf{\Lambda}}^{{G}}=\mathsf{\Lambda} \times \{E^u\}$ with
stable index
$$
d^{{G}}_{ss}\eqdef\mathrm{ind}^s({\Lambda}^{{G}})=ss+\dim
G(u,m)=ss+u(m-u).
$$
 {Observe that shrinking $\mathcal{C}^G$ if necessary, the contraction of $D\Phi_\ell$ on $\mathcal{C}^G$ dominates the contraction
of $F$ (that is $\beta\nu<\nu$). Then, the width $\alpha>0$ of the
stable cone-field  on $\mathsf{R}\times \mathcal{C}^G$ of $F^G$ is
the same width $\alpha$ that we have for the stable cone-field  on
$\mathsf{R}$ of $F$.}

Since $\beta\nu < \nu<\lambda$ then
$F^{{G}}|_{{\mathsf{\Lambda}}^{{G}}}$ dominates the fiber dynamics
given by $\phi_1,\dots,\phi_\kappa$. By Theorem~\ref{thmBKR}, we
have that $\Gamma^{{G}}=\Gamma \times \{E^u\}$ is a $cs$-blender
of  {central dimension} $c>0$  of $\Phi^{{G}}$.
 {Moreover, the family of
{$( {\alpha,}\nu,\delta)$}-horizontal, $d^{{G}}_{ss}$-dimensional
$C^1$-discs in ${\mathcal{B}}^{{G}}=\mathcal{B} \times
\mathcal{C}^{{G}}$ is a superposition region of the blender
$\Gamma^{{G}}$. Here, as in~\S\ref{sec:affine-blender},
$0<\delta<\lambda L /2$ whereas $L$ is the Lebesgue number of the
cover~\eqref{eq:cover}.}

\subsection{Folding manifold with respect to the affine blender}
\label{sec:folding-manifold} Next we introduce the notion of a
folding manifold.
To do this,
%
we will consider a submanifold $\mathcal{S}$ of $\mathbb{R}^m$ of
dimension $ss+c$.
In what follows we  {identify canonically the tangent space
 $T_z\mathcal{S}$ of $\mathcal{S}$ at $x$ with a} subspace of
$\mathbb{R}^m$.

\begin{figure}
 \footnotesize
  \begin{center}
   \begin{picture}(400,220)
\put(100,0){
\includegraphics[scale=1]{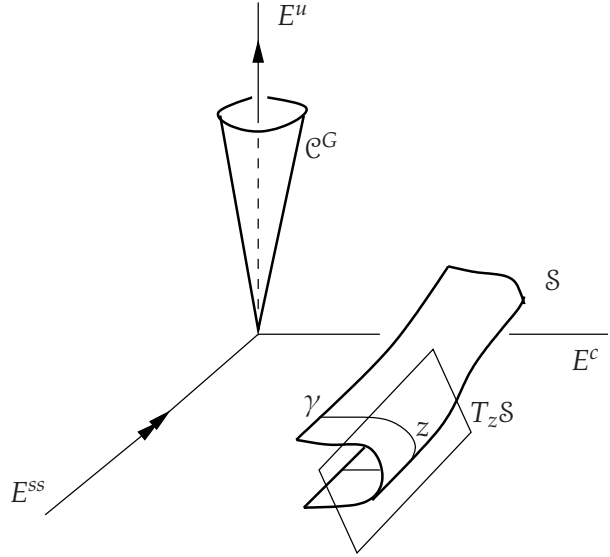}}
\put(200,150){\normalsize $\mathcal{C}^{{G}}$}
\put(290,100){\normalsize $\mathcal{S}$} \put(190,200){\normalsize
$E^{u}$} \put(90,20){\normalsize $E^{ss}$}
\put(300,70){\normalsize $E^{c}$} \put(242,45){\normalsize $z$}
 \put(262,50){\normalsize $T_z\mathcal{S}$}
 \put(200,55){\normalsize $\gamma$}
 \end{picture}
 \end{center}
 \caption{
 Folding manifold with respect to $\mathcal{B}^{{G}}=\mathcal{B}\times
 \mathcal{C}^{{G}}$.
 }\label{fig1}
 \end{figure}

\begin{defi} \label{def:dobra}
We say that $\mathcal{S}$ is a {$( {\alpha,}
\nu,\delta)$}-\emph{folding $C^r$-manifold} with respect to
$\mathcal{B}^{{G}}=\mathcal{B}\times \mathcal{C}^{{G}}$
if  {there is $\epsilon>0$ such that}
\begin{enumerate}[leftmargin=0.45cm,itemsep=0.2cm]
\item {$\mathcal{S}$ is parameterized as a
 $(ss+c)$-dimensional $C^r$-embedding
$\mathcal{S}:  [-2,2]^{ss}\times  [-\epsilon,\epsilon]^c \to
   \overline{\mathcal{B}}$, of the form}
   $$   {\mathcal{S}(x,t)=(x,(t_1,\dots,t_u),h(x,t))\in \mathbb{R}^{ss}\times \mathbb{R}^{u}\times\mathbb{R}^c}$$
    {with $x\in [-2,2]^{ss}$ and $t=(t_1,\dots,t_u,\dots,t_c)\in
   \mathbb{R}^c$;}
\item there is $y\in B$ such that
$d(h(x,t),y)<\delta$ for all $(x,t)\in [-2,2]^{ss}\times
 [-\epsilon,\epsilon]^c$;
\item \label{foldingCu}{for all $x\in [-2,2]^{ss}$ and
$E\in  {\overline{\mathcal{C}^G}}$ there is a unique $t\in
[-\epsilon,\epsilon]^c$ such that $E$ is a subspace of $T_z
\mathcal{S}$ with $z= \mathcal{S}(x,t)$. Moreover, $t=t(x,E)$
varies $C^{r-1}$-continuously} with $(x,E)$ and
$$
   (\|Dh\|_{\infty} \cdot \max\{1,\|Dt\|_{\infty}\})\cdot \nu <\delta
   \quad  {\text{and} \quad \|Dt\|_\infty \leq \alpha}.
$$
\end{enumerate}
\end{defi}
{Let us explain geometrically the above notion of a folding
manifold. First define the unstable cone for some small $\theta>0$
as
\begin{align*}
\mathcal{C}^{u}_\theta&=\{(u,v,w)\in
\mathbb{R}^m=\mathbb{R}^{ss}\oplus \mathbb{R}^u \oplus
\mathbb{R}^c: \ \ \|u+w\| <  \theta \|v \|\}\cup\{0\}.
\end{align*}
Each vector subspace $E$ of dimension $u$ contained in the cone
$\mathcal{C}^{u}_\theta$ can be identified with an element of
$\mathcal{C}^{G}$ and vice-versa. Then condition~\eqref{foldingCu}
implies that for every $x\in [-2,2]^{ss}$,
\begin{equation}
\label{eq:cover-Cu}
   \overline{\mathcal{C}^{u}_\theta} \subset \bigcup_{t\in [-\epsilon,\epsilon]^c}
   T_{\gamma(t)}\mathcal{S} \quad \text{where $\gamma(t)=\mathcal{S}(x,t)$.}
\end{equation}
In fact, the uniqueness in condition~\eqref{foldingCu} implies the
injectivity of the map $t\mapsto T_{\gamma(t)}\mathcal{S}$ and
thus, the parameters $\epsilon$ and $\theta$ can be interpreted as
the size of the neighborhood $\mathcal{C}^G$ of
$E^u=\{0^{ss}\}\times \mathbb{R}^u \times \{0^c\}$ in $G(u,m)$. }



{Next we will show an example of a $(
{\alpha,}\nu,\delta)$-folding manifold with respect to
$\mathcal{B}^G=\mathcal{B}\times \mathcal{C}^G$. Recall that $B$
is an open set of $\mathbb{R}^c$ containing the origin.
 {Up to a conjugacy with a translation}, we can
assume that $\mathcal{B}=\mathsf{R}\times B$ contains
$(-2,2)^{ss}\times \{0^u\} \times \{0^c\}$. }

\begin{exap} \label{exap:dobra} {
Consider the $(ss+c)$-dimensional embedding given by
$$
  \mathcal{S}:[-2,2]^{ss}\times [-\epsilon,\epsilon]^c \to
  \mathbb{R}^m, \ \ \mathcal{S}(x,t)=(x,(t_1,\dots,t_u),
  ( {H}(t), {t_{u+1},\dots,t_c}))
  \in \mathbb{R}^{ss}\times\mathbb{R}^u\times \mathbb{R}^c
$$
where $x\in \mathbb{R}^{ss}$, $t=(t_1,\dots,t_u,t_{u+1},\dots,t_c)
\in [-\epsilon,\epsilon]^c$ and $ {H}(t)=( {H}_1(t),\dots,
{H_u}(t))$ with
\begin{align*}
    {H}_i(t)=  \sum_{j=0}^{u-1} t_{j+1} t_{ju+i}  \quad \text{for $i=1,\dots,u$.}
\end{align*}
For a fixed $\delta>0$, we will prove that $\mathcal{S}$ is a $(
{\alpha,}\nu,\delta)$-folding  {$C^\infty$}-manifold
 {for any $\alpha>0$ large enough} and
$\epsilon,\nu>0$ small enough. To do this, we will show that
$\mathcal{S}$ satisfies all the conditions of
Definition~\ref{def:dobra}.}

{It is straightforward that $\mathcal{S}$ is a
$(ss+c)$-dimensional  {$C^\infty$}-embedding. Since $\mathcal{B}$
contains $(-2,2)^{ss}\times \{0^u\} \times \{0^c\}$, then
$\mathcal{S}([-2,2]^{ss}\times [-\epsilon,\epsilon]^c)\subset
\mathcal{B}$ for any $\epsilon>0$ small enough, concluding the
first condition in Definition~\ref{def:dobra}. Observe now that
the central coordinate of $\mathcal{S}$, i.e., the map
$h(t)=\mathscr{P} \circ \mathcal{S}(x,t)$ does not depend on $x$.
Moreover, if $\epsilon>0$ is small enough then
$d(h(t),0^c)<\delta$ for all $(x,t)\in [-2,2]^{ss}\times
[-\epsilon,\epsilon]^c$, as is required by the second condition in
Definition~\ref{def:dobra}.} {To conclude that $\mathcal{S}$ is a
folding manifold, it only remains to prove the last condition in
Definition~\ref{def:dobra}, which is somewhat longer and will be
done in the next paragraphs.

By a direct computation, $T_z\mathcal{S}$ at $z=\mathcal{S}(x,t)$
is given by
$$
  T_z\mathcal{S}(x',t')=(x',(t'_1,\dots,t'_u),(g(t',t), {t'_{u+1},\dots,t'_{c}}))\in
  \mathbb{R}^{ss}\times \mathbb{R}^u\times \mathbb{R}^c
$$
where $x'\in \mathbb{R}^{ss}$,
$t'=(t'_1,\dots,t'_u,t'_{u+1},\dots,t'_c)\in \mathbb{R}^c$ and
$g(t',t)=(g_1(t',t),\dots,g_{ {u}}(t',t))$ with
\begin{align*}
   g^{}_i(t',t)=  \sum_{j=0}^{u-1}  t'_{j+1} t^{}_{ju+i} + t^{}_{j+1}
   t'_{ju+i}  \quad \text{for $i=1,\dots,u$.}
\end{align*}
We want to prove that for any $x\in [-2,2]$ and $E\in\overline{
\mathcal{C}^G}$, there is a unique $t\in [-\epsilon,\epsilon]^c$
such that $E$ is a subspace of $T_z\mathcal{S}$ for
$z=\mathcal{S}(x,t)$. Observe that if $E=\langle v_k:
k=1,\dots,u\rangle$ is generated by linearly independent vectors
$v_k$, then $E$ is a subspace of $T_z\mathcal{S}$ if and only if
$v_k\in T_z\mathcal{S}$ for all $k=1,\dots,u$. Denoting
$v_k=(a_k,b_k,c_k)\in \mathbb{R}^{ss}\times \mathbb{R}^u\times
\mathbb{R}^c$, the above condition is equivalent to the existence
of $t\in [-\epsilon,\epsilon]^c$ such that for every
$k=1,\dots,u$,  {there are $x_k'\in\mathbb{R}^{ss}$, and
$t'_k\in\mathbb{R}^c$ satisfying:}
$$
    x'_k=a^{}_k \ \ \ \ t'_k=
    (b^{}_{k1},\dots,b^{}_{ku},c^{}_{k\,u+1},\dots,c^{}_{kc}) \ \ \ \text{and}
    \ \  c^{}_{ki}= g^{}_i(t'_k,t)  \ \ \ \text{for $i=1,\dots,u$}.
$$
Notice that $g^{}_i(t'_k,t)$ can be written as a scalar product of
$t$ by a vector  {$\vec{a}_{ki}\in \mathbb{R}^c$} that depends
on~$t'_k$. Thus,  {having into account that $c=u^2$,} we can write
the relation $c_{ki}=g_i(t'_k,t)$ for $k=1,\dots,u$ and
$i=1,\dots,u$ as a matrix product
$$
At=\vec{c} \ \ \text{where} \ \vec{c}=(c_{11},\dots,c_{1
{u}},\dots,c_{u1},\dots,c_{u {u}})^T
$$
and $A
{=[\vec{a}_{11};\dots;\vec{a}_{1u};\dots;\vec{a}_{u1};\dots;\vec{a}_{uu}]}$
is a $c$-by-$c$ matrix that depends on $t'_k$ for $k=1,\dots,u$.
In fact, since $t'_k$ form part of the coordinates of the vector
$v_k$, then $A=A(E)$ depends on the vector space $E$. Similarly
this holds for $\vec{c}=\vec{c}(E)$. Hence, to find the required
$t\in [-\epsilon,\epsilon]^c$ we only need to show that the linear
system $A(E)\cdot t = \vec{c}(E)$ is uniquely solved.}

{ To do this, we will analyze the determinant of $A$ at
$E^u=\{0^{ss}\}\times\mathbb{R}^u\times\{0^c\}$ in which the open
set $\mathcal{C}^G$ is centered. Observe that $E^u$ is generated
by the vectors $e^u_k=(0^{ss},e_k,0^c)$ for $k=1,\dots,u$, where
$e_k$ is the $k$-th canonical vector in $\mathbb{R}^u$.  Thus,
$t'_k=(e_k,0^{c-u})$ and then
$$
g^{}_i(t'_k,t)=t_{(k-1)u+i}+\delta_{ki} \, t_1 \ \ \text{for} \ \
k=1,\dots,u \ \ \text{and} \ \ i=1,\dots,u,
$$ where $\delta_{ki}$ is the Kronecker delta. In view
of this, $A(E^u)=\mathrm{Id}+\mathrm{L}$ where $\mathrm{Id}$ is
the identity matrix and $\mathrm{L}$ is a matrix whose first
column is given by $(e_1,e_2,\dots,e_u)^T$ and the rest of the
elements are zero. Hence $A(E^u)$ is a triangular matrix with
$\det A(E^u) =2 \not =0$. Thus, we get that $A(E)\cdot t =
\vec{c}(E)$ is uniquely solved for any $E$ close enough to $E^u$.}

{ This shows the first part of the last condition in
Definition~\ref{def:dobra} but still we need to prove that $C \nu
<\delta$ where $C=\|Dh\|_\infty \cdot \max\{1,\|Dt\|_\infty\}$
 {and $\|Dt\|_\infty\leq \alpha$}. Since both
$h=h(t)$ and $t=t(E)$ are functions of class $C^{ {\infty}}$, then
$C<\infty$ over $[-2,2]^{ss}\times \mathcal{C}^{G}$ and
$\|Dt\|_\infty<\infty$. Thus, this condition trivially holds by
taking $\nu>0$ small enough  {and $\alpha$ large enough.}}
\end{exap}

\begin{rem}
{ In Proposition~\ref{prop1:appendix} in Appendix~\ref{appendix}
we show that actually $\mathcal{S}$ in the above example is a $(
{\alpha,}\nu,\delta)$-folding $ {C^\infty}$-manifold for any $\nu<
\delta$  {and $\alpha>1$}.}
\end{rem}

\begin{rem} \label{rem:rob-dobra}
{Fixing a small enough $\epsilon>0$, for which $\mathcal{S}$ is a
$( {\alpha,}\nu,\delta)$-folding manifold, the above example is
$C^2$-robust in the following sense. Let $\tilde{\mathcal{S}}:
[-2,2]^{ss}\times  [-\epsilon,\epsilon]^c \to
   \overline{\mathcal{B}}$
be $(ss+c)$-dimensional $C^2$-embedding, which is $C^2$-
sufficiently close to $\mathcal{S}$. Then
 $\tilde{\mathcal{S}}$ is also a $( {\alpha,}\nu,\delta)$-folding manifold.
Let us comment on why this is true. What has to be shown is
basically condition~\eqref{foldingCu} of
Definition~\ref{def:dobra}. The function $\tilde{\mathcal{S}}$ can
be written in the form
$$
(x,t)\mapsto (x,(t_1,\dots,t_u),(H(x,t),t_{u+1},\dots,t_c)), \
\text{where} \ H(x,t)=H(t)+\kappa(x,t).
$$
Here $H(t)$ comes from the embedding $\mathcal{S}$, and
$\kappa(x,t)$ is a small $C^2$-perturbation. Following the
notation of the previous example, the non-linear equations that
now have to be solved take the form
$$
x'_k=a_k, \ \ t'_k=(b_{k1},\dots,b_{ku},c_{k\, u+1},\dots,c_{kc}),
\ \text{and} \ c_{ki}=g_i(t'_k,t)+\varkappa_{i} \quad \text{for
$i,k=1,\dots,u$}
$$
where $\varkappa_{i}$ are functions which depend on $(x,t,x',t')$
with small $C^1$-derivative.
The solution of these equations for $t$ given $x$ and $E$, is then
guaranteed by an application of the Implicit Function Theorem. }
\end{rem}

Let $\mathcal{S}$ be a $( {\alpha,}\nu,\delta)$-folding
$C^r$-manifold with respect to
$\mathcal{B}^{{G}}=\mathcal{B}\times\mathcal{C}^{{G}}$.
Consider
$$
    \mathcal{S}^{{G}}  = \{ (z,E): z\in \mathcal{S}, \  E \in G(u,m)
    \ \
    \text{and} \  \ {E  {\subset} T_z\mathcal{S}}  \} \subset
    G_u(\mathbb{R}^m).
$$
One can see $\mathcal{S}^{{G}}$ as a fiber bundle over
$\mathcal{S}$ with fibers
$$
    (\mathcal{S}^{{G}})_z=\{E\in G(u,m): E  {\subset}  T_z\mathcal{S} \}.
$$
Notice that $(\mathcal{S}^{{G}})_z$ is a compact manifold of
dimension $ \dim G(u,ss+c)= u(ss+c-u)$. Then, the dimension of
$\mathcal{S}^{{G}}$ is $ss+c + u(ss+c-u) =  {ss+}
 u(m-u)+c-u^2$. In fact, since $c=u^2$ we have that this dimension
coincides with $d^{{G}}_{ss}=ss+u(m-u)$.

\begin{lem} \label{lem-induced-folding-manifold}
The set $\mathcal{H}^{{G}}=\mathcal{S}^{{G}} \cap
\overline{\mathcal{B}^{{G}}}$ is a {($
{\alpha,}\nu,\delta)$}-horizontal $d^{{G}}_{ss}$-dimensional
$C^{{r-1}}$-disc in $\mathcal{B}^{{G}}$.
\end{lem}

\begin{proof} First notice that $\mathcal{H}^{{G}}$ is
a $d^{{G}}_{ss}$-dimensional $C^{{r-1}}$-disc in
$\mathcal{B}^{{G}}$. This follows from Definition~\ref{def:dobra},
since given any $x\in [-2,2]^{ss}$ and any $E\in
\overline{\mathcal{C}^{{G}}}$ we have a unique $t=t(x,E)\in
[{-\epsilon,\epsilon}]^c$ which varies $C^r$-continuously with $x$
and $E$ such that
$E\subset  T_{z}\mathcal{S}$ with $z=\mathcal{S}(x,t)$. {Thus we
can parametrize $\mathcal{H}^G$ as}
$$
  \mathcal{H}^{{G}}: [-2,2]^{ss}\times \overline{\mathcal{C}^{{G}}} \longrightarrow
  G_u(\mathbb{R}^m), \qquad \mathcal{H}^{{G}}(x,E)=(\mathcal{S}(x,t), E)
  \in \overline{\mathcal{B}}\times \overline{\mathcal{C}^{{G}}}=
  \overline{\mathcal{B}^{{G}}}
$$
is a $C^{{r-1}}$-disc in $\mathcal{B}^{{G}}$. On the other hand,
the  {unstable and} central  {coordinates} of this disc
 {are} given by
\begin{align*}
  {g^{{G}}(x,E)}& {=\mathscr{P}_u \circ
 \mathcal{H}^{{{G}}}(x,E)=g(t(x,E)), \ \ \text{and}} \\
  h^{{G}}(x,E)&=\mathscr{P}_{ {c}} \circ
 \mathcal{H}^{{{G}}}(x,E)=h(x,t(x,E))  \qquad \text{for} \ \
 (x,E)\in [-2,2]^{ss}\times\overline{\mathcal{C}^{{G}}},
\end{align*}
where
$$
 {(t_1,\dots,t_u)=g(t)=\mathscr{P}_u\circ
\mathcal{S}(x,t) \quad \text{and}} \quad h(x,t)=\mathscr{P}_{
{c}}\circ \mathcal{S}(x,t)
$$
 {are} the  {unstable and} central
 {coordinates} of the folding manifold
$\mathcal{S}$. Here $\mathscr{P}_u$ and $\mathscr{P}_c$ denote
the {canonical} projections on {$\mathbb{R}^u$ and} $\mathbb{R}^c$
 {respectively}. Hence, again, by the definition of
folding manifold we have $y\in B$ such that
$d(h^{{G}}(x,E),y)=d(h(x,t(x,E)),y)<\delta$,
$$
  \|Dh^{{G}}\|_{\infty} \leq \|Dh\|_{\infty}
  \max \{1, \|Dt\|_{\infty}\}=C \quad \text{with \ \ $C\nu < \delta$.}
$$
and  {$$
  \|Dg^G\|_\infty \leq \|Dg\|_\infty \|Dt\|_\infty \leq
  \|Dt\|_\infty \leq \alpha.
$$}
 This proves that $\mathcal{H}^{{G}}$ is
$( {\alpha,}\nu,\delta)$-horizontal disc concludes the proof.
\end{proof}

\begin{rem}
The {$( {\alpha,}\nu,\delta)$}-horizontal $C^{{r-1}}$-disc
$\mathcal{H}^{{G}}$ obtained from the $C^r$-folding manifold in
Example~\ref{exap:dobra} is $C^0$-close to a horizontal disc but
$C^1$-far from it.
\end{rem}

\subsection{Robust tangencies with a folding manifold}
\label{sec:robust-tangency-folding-manifold} {Recall that the
$C^r$-diffeomorphism $\Phi$ of $\mathbb{R}^m$ we are considering
in this section was introduced
in~\S\ref{sec:affine-blenders-grasmannian}. This map has a
$cs$-blender $\Gamma$ of  {central dimension} $c>0$, where a
superposition region contains the family of $(
{\alpha,}\nu,\delta)$-horizonal discs in $\mathcal{B}$ with
$\delta<\lambda L/2$.} Now, we will prove the following key
result:
\begin{prop}
\label{prop-tangency-folding} There is a {$C^2$-neighborhood}
$\mathscr{U}$ of $\Phi$ such that {for any $(
{\alpha,}\nu,\delta)$-folding $C^r$-manifold $\mathcal{S}$ with
respect to $\mathcal{B}^G=\mathcal{B}\times \mathcal{C}^G$ it
holds that}
 {for any} $g\in \mathscr{U}$
there are points $z\in\Gamma_g$ and $x\in W^u_{loc}(z) \cap
\mathcal{S}$ such that
\begin{equation}
\label{eq:tangencia}
  \dim T_x W^u_{loc}(z) \cap T_x \mathcal{S} = u
  \quad \text{or equivalently, \ \ $T_xW^u_{loc}(z) \subset T_x
  \mathcal{S}$}.
\end{equation}
 {In particular, since the codimension of
$W^u_{loc}(z)$ coincides with the dimension of $\mathcal{S}$,
these two manifolds intersect at $x$ in a tangency of codimension
$u$.}
\end{prop}

\begin{proof} We recall that $\Gamma^{{G}}=\Gamma\times\{E^u\}$
is a $cs$-blender of  {central dimension} $c>0$ for the induced
$C^{1}$-diffeomorphism $\Phi^{{G}}$, whose superposition region
contains the set $\mathscr{D}$ of $(
{\alpha,}\nu,\delta)$-horizontal $d^{{G}}_{ss}$-dimensional
$C^1$-discs in $\mathcal{B}^{{G}}=\mathcal{B} \times
\mathcal{C}^{{G}}$.
Hence, {by definition of a blender}, there is a $C^1$-neighborhood
$\mathscr{U}^{{G}}$ of $\Phi^{{G}}$
where for {each map $\Psi \in \mathscr{U}^G$} we have an
 {intersection} between  {{each} disc
in $\mathscr{D}$ and the local} unstable manifold of
the continuation of $\Gamma^{{G}}$ for $\Psi$. We take a
{$C^2$-neighborhood} $\mathscr{U}$ of $\Phi$ so that for every
$g\in\mathscr{U}$ its induced {$C^1$-diffeomorphism} $g^{{G}}$ on
$G_u(\mathbb{R}^m)$ belongs to $\mathscr{U}^{{G}}$. Hence, {the
continuation $\Gamma^{G}_g$ of $\Gamma^G$ for $g^{{G}}$ is a
$cs$-blender}. Moreover, $W^u_{loc}(\Gamma_g^{{G}})$ is {
{laminated by}  plaques of dimension $u$ which project one-to-one
onto $W^u_{loc}(\Gamma_g)$}. In particular,
 {shrinking $\mathscr{U}$ if necessary},
\begin{equation}\label{eq:implication}
\text{if $(x,E)\in W^u_{loc}(\Gamma_g^{{G}})$ then there is $z\in
\Gamma_g$ such that $x\in W^u_{loc}(z)$ and $E=T_xW^u_{loc}(z)\in
\mathcal{C}^{{G}}$.}
\end{equation}
On the hand, {if $\mathcal{S}$ is a $(
{\alpha,}\nu,\delta)$-folding manifold with respect to
$\mathcal{B}^G$}, then by
Lemma~\ref{lem-induced-folding-manifold}, the manifold
$\mathcal{S}^{{G}}$ contains a ${(
{\alpha,}\nu,\delta)}$-horizontal $d^{{G}}_{ss}$-dimensional
$C^1$-disc $\mathcal{H}^{{G}}$ in $\mathcal{B}^{{G}}$. Hence
$\mathcal{H}^{{G}} \in \mathscr{D}$. Thus,
$W^u_{loc}(\Gamma_g^{{G}})\cap \mathcal{H}^{{G}} \not =
\emptyset$. Consequently, there is $(x,E)$ belonging  to
$\mathcal{H}^{{G}}\subset \mathcal{S}^{{G}}$ and
$W^u_{loc}(\Gamma_g^{{G}})$. In particular,
from~\eqref{eq:implication}, we get that  $x\in \mathcal{S}\cap
W^u_{loc}(\Gamma_g)$ and $T_xW^u_{loc}(z)=E\subset T_x
\mathcal{S}$ for some $z\in \Gamma_g$. This completes the proof.
\end{proof}

\subsection{Proof of Theorem~\ref{thmD}}
\label{sec:prove-ThmD} Finally we prove Theorem~\ref{thmD} by
assuming that the global stable manifold of a periodic point $P$
in the affine $cs$-blender $\Gamma$ contains {the folding manifold
with respect to
$\mathcal{B}^{{G}}=\mathcal{B}\times\mathcal{C}^{{G}}$ given in
Example~\ref{exap:dobra}. As was explained in
Remark~\ref{rem:rob-dobra} this folding manifold is $C^2$-robust.}
Thus,  the stable manifold $W^s(P_g)$ of the continuation $P_g$ of
$P$ contains a folding manifold with respect to
$\mathcal{B}^{{G}}$ for all small enough $C^2$-perturbations $g$
of $\Phi$. Then, Proposition~\ref{prop-tangency-folding} implies
that  {there is $z\in \Gamma_g$ such that $W^u_{loc}(z)$ and
$W^s(P_g)$ have a tangency of codimension $u>0$}. Thus, we get
that $\Phi$ has a $C^2$-robust homoclininc tangency of codimension
$u$. Moreover, using ~\eqref{eq:tangencia} we can conclude that
the tangency  {cannot be inside a strong partially hyperbolic set.
To see this, notice that $T_x\mathcal{S}=E^{ss}\oplus F$ where
$E^{ss}=\mathbb{R}^{ss}\times\{0^{u}\}\times\{0^c\}$,
$T_xW^u_{loc}(z)\subset F$ and $T_xW^u_{loc}(z)\in \mathcal{C}^G$.
Oberve that $E^{ss}$ is in the strong stable cone-field on
$\mathcal{B}$ of $\Phi$, while $\mathcal{C}^G$ can be identified
with the unstable cone on $\mathcal{B}$ of $\Phi$. Since both
cone-fields are disjoint, we obtain that $T_xW^u_{loc}(z)$ cannot
be the strong stable direction. This proves that the tangency must
be inside a weak partially hyperbolic set.} Finally, recall that
$c=u^2$ and then $m=ss+c+u>
u^2+u$, completing the proof. 

\section{ {Degenerate} unfoldings of  {tangencies}}
\label{sec:unfoding}
 { Recall that by a tangency we understand the
opposite of a transverse intersection.}
We will introduce the notion of {degenerate}
unfoldings 
of a  {tangency}
between two submanifolds $\mathcal{L}_0$  {and} $\mathcal{S}_0$.
 Let $\mathcal{L}=(\mathcal{L}_a)_a$ and
$\mathcal{S}=(\mathcal{S}_a)_a$ be  $k$-parameter families of
submanifolds $\mathcal{L}_a$ and $\mathcal{S}_a$ of $\mathcal{M}$
diffeomorphic to $\mathcal{L}_0$ and $\mathcal{S}_0$ respectively
by families of diffeomorphisms $C^{d,r}$-close to the identity
 {with $0<d\leq r$.}

\begin{defi}
We say that $\mathcal{L}$  {and} $\mathcal{S}$ has a
 {tangency} 
at $a=0$
which unfolds \emph{$C^d$-{degenerate}} if there exist $x=(x_a)_a,
y=(y_a)_a \in C^d(\mathbb{I}^k,\mathcal{M})$ such that
$$
    x_a \in \mathcal{L}_a \  \ \text{and} \  \ y_a \in \mathcal{S}_a  \ \
    \text{so that}
    \ \
    d(x_a,y_a)=o(\|a\|^{ {d}}) \ \ \text{at $a=0$}.
$$
\end{defi}

A useful formalism to define a $C^d$-{degenerate} unfolding of a
 {tangency}  is to consider the spa\-ce of jets
$J^d_0(\mathbb{I}^k,\mathcal{M})$ whose elements are the
coefficients of the truncated~Taylor~series~at~$a=0$,
$$
   J_0^d(z)=(z_a,\partial^1_a z_a,\partial^2_a
   z_a,\dots,\partial^d_az_a)_{|_{\,a=0}}
   \ \ \text{with} \ \ z=(z_a)_a \in
   C^d(\mathbb{I}^k,\mathcal{M}).
$$
For a more precise definition see~\S\ref{sec:jets}.  Then
$$
   d(x_a,y_a)=o(\|a\|^{{ {d}}}) \ \ \text{at $a=0$} \quad \text{if and only
   if} \quad J^d_0(x)=J_0^d(y).
$$
The set $J_0^d(\mathbb{I}^k,\mathcal{M})$ can be endowed with a
smooth manifold structure sometimes called the \emph{manifold of
$(d,k)$-velocities over $\mathcal{M}$}.


Next, we will be interested in unfoldings which control not only
the separation of points on the manifold, but also the separation
of the tangent spaces.  {Hence, to control this separation, we
assume $d<r$ and introduce the following definition.}
Let $G_{ {\ell}}(\mathcal{M})$ be the $ {\ell}$-th Grassmannian
bundle of $\mathcal{M}$. That is, the fiber bundle over
$\mathcal{M}$ whose fibers are the $ {\ell}$-th Grassmannian
manifold of the tangent space $T_p\mathcal{M}$, i.e.,
$$
    G_{ {\ell}}(\mathcal{M}) = \bigsqcup_{p\in \mathcal{M}} G_ { {\ell}}(\mathcal{M})_p =
    \bigcup_{p\in \mathcal{M}} \{p\}\times G({ {\ell}},T_p\mathcal{M})
$$
where $G({ {\ell}},T_p\mathcal{M})$ is the set of all ${
{\ell}}$-dimensional linear subspaces of $T_pM$.


\begin{defi} \label{def:tangencia}
We say that $\mathcal{L}$  {and} $\mathcal{S}$ has a tangency of
 {dimension $\ell>0$} at $a=0$ which unfolds
$C^d$-{degenerate} if there exist $x=(x_a)_a, y=(y_a)_a \in
C^d(\mathbb{I}^k,G_{ {\ell}}(\mathcal{M}))$ such that
$$
    x_a \in G_{ {\ell}}(\mathcal{L}_a) \  \ \text{and} \  \ y_a \in G_{ {\ell}}(\mathcal{S}_a)  \ \ \text{so that}
    \ \
    d(x_a,y_a)=o(\|a\|^{ {d}}) \ \ \text{at $a=0$.}
$$
Using the formalism of jets, the unfolding is $C^d$-{degenerate}
if and only if $
 J^d_0(x)=J_0^d(y)$.
\end{defi}
\begin{rem}
In the terminology of~\cite{Ber16}, $C^d$-{degenerate} unfoldings
of a tangency  {of dimension one (between curves in dimension
two)} are called \emph{$C^d$-paratangencies.}
\end{rem}


\section{Parablenders}
\label{sec:parablender} The concept of parablender was initially
introduced by Berger~\cite{Ber16} for endomorphisms (see
also~\cite{BCP17,Ber17,Ber19}). The following generalizes both,
the blender (Definition~\ref{def:blender}) and the definition of
parablender for diffeomorphisms given  {also by Berger}
in~\cite[Example~1.21,Def.~1.23]{Ber19}.

%

\begin{defi} \label{def:parablender}
Let $\Gamma_0$ be a $cs$-blender  {of central dimension} $c\geq 1$
 {and strong stable dimension
$d_{ss}=\mathrm{Ind}^s(\Gamma_0)-c$} of a $C^r$-diffeomorphism
$f_0$ of $\mathcal{M}$. Consider a $k$-parameter $C^d$-family
$f=(f_a)_a$ of $C^r$-diffeomorphisms of $\mathcal{M}$  unfolding
$f_0$ at $a=0$. A~family $\Gamma=(\Gamma_a)_a$ of compact sets
$\Gamma_a$ of $\mathcal{M}$ is said to be \emph{
{$C^d$-}$cs$-parablender  {at $a=0$} of  {central dimension} $c$}
for $f$ if
\begin{enumerate}[itemsep=0.1cm]
\item $\Gamma_a$ is the continuation for $f_a$
of the $cs$-blender $\Gamma_0$  for all $a\in\mathbb{I}^k$,
\item there exists an open set $\mathscr{D}$ of  $k$-parameter
$C^d$-families $\mathcal{D}=(\mathcal{D}_a)_a$,
 {where each $\mathcal{D}_a$ is a $C^r$-embedded
$d_{ss}$-dimensional disc into $\mathcal{M}$,}
\item there exists a $C^{d,r}$-neighborhood $\mathscr{U}$ of
$\Phi$,
\end{enumerate}
such that for every $g=(g_a)_a \in \mathscr{U}$ and
$\mathcal{D}=(\mathcal{D}_a)_a \in \mathscr{D}^{ss}$ it holds that
$$
\text{$W^u_{loc}(\Gamma_g)$  {and} $\mathcal{D}$ has a
 {tangency} at $a=0$  which unfolds
$C^d$-{degenerate}.}
$$
That is, there are $x=(x_a)_a$, $y=(y_a)_a$, $z=(z_a)_a\in
C^d(\mathbb{I}^k,\mathcal{M})$ with
$$
z_a \in \Gamma_{a,g}, \ \ x_a \in W^u_{loc}(z_a) \ \ \text{and} \
\ y_a\in \mathcal{D}_a \ \ \text{such that} \ \
d(x_a,y_a)=o(\|a\|^{ {d}}) \ \ \text{at $a=0$}
$$
where $\Gamma_{a,g}$ is the  continuation for $g_a$ of the
$cs$-blender $\Gamma_a$ for all $a\in \mathbb{I}^k$.

A \emph{$C^d$-$cu$-parablender}  {at $a=0$}  of
 {central dimension $c$} is
 {$C^d$-}$cs$-parablender  for
$f^{-1}=(f^{-1}_a)^{}_a$.
\end{defi}

\begin{rem}
For $k=0$ and $d=r=1$, i.e., when there are no parameters and the
class is $C^1$, the above definition of a parablender coincides
with the definition of a blender.  {As was mentioned after
Definition~\ref{def:blender}, the tangency between
$W^u_{loc}(z_0)$ and $\mathcal{D}_0$ has codimension at least $c$.
In general, it is a quasi-transverse intersection of codimension
exactly $c$.}
\end{rem}


\begin{rem}
For simplicity, to introduce parablenders, we have chosen the
parameter $a=0$. However, we  can also define a parablender at any
other parameter $a=a_0$ with $a_0\in \mathbb{I}^k$. Moreover, we
will say that a $k$-parameter family $f=(f_a)_a$ has a parablender
$\Gamma=(\Gamma_a)_a$ at \mbox{\emph{any parameter}} when $\Gamma$
is a parablender for $f$ at $a=a_0$ for all $a_0\in \mathbb{I}^k$
with $\mathscr{D}$ and $\mathscr{U}$ independent of the value
$a_0$.
\end{rem}

Parablenders are a mechanism to provide $C^{d,r}$-open sets of
families of diffeomorphisms which are $C^d$-{degenerate}
unfoldings of  {tangencies (of dimension zero in general)}. The
following theorem proves the existence of such open sets.

\begin{thm} \label{thm-parablender}
Any manifold  of dimension $m> c+1$ admits a $k$\,-\,parameter
$C^d$-family of $C^r$-diffeomorphisms with { $0<d < r$} having a
 {$C^d$-}parablender of  {central
dimension $c\geq 1$} \mbox{at any parameter.}
\end{thm}

We split the proof of this theorem into several parts.
 {The basic idea to obtain a parablender is by
constructing a blender for the induced dynamics $\widehat{f}$ in
the space of jets using a parametric family $f=(f_a)_a$ of
diffeomorphisms. To do this, first in~\S\ref{sec:jets}, we
introduce the jet space and the induced dynamics. After that, we
consider in~\S\ref{sec:affine-blenders-family} a class of
parametric families of diffeomorphisms with a family
$\Gamma=(\Gamma_a)_a$, where each $\Gamma_a$ is an affine
$cs$-blender as constructed in~\S\ref{sec:affine-blender}. In
order to see that this family of blenders is, indeed, a
parablender we show in~\S\ref{sec:blender-jets} that the induced
dynamics on the jet space has a $cs$-blender $\widehat{\Gamma}$.
To conclude that $\Gamma$ is a parablender we also need to provide
an open set $\mathscr{D}$ of $k$-parametric families of discs.
This is done in~\S\ref{sec:discos-afin-jet-blender} where
additionally we show that each family of discs
$\mathcal{D}=(\mathcal{D}_a)_a$ in $\mathscr{D}$  induces a disc
$\widehat{D}$ of jets in a superposition region of the blender
$\widehat{\Gamma}$. Finally, we show
in~\S\ref{sec:parablender-from-blender} that the robust
intersection between each of these discs $\widehat{\mathcal{D}}$
of jets and the local unstable manifolds of $\widehat{\Gamma}$
implies a degenerate unfolding of a tangency between the family
$\mathcal{D}$ and the local unstable manifold of  $\Gamma$ in the
sense of Definition~\ref{def:parablender}.}

First of all, notice that we will provide the existence of
parablenders by constructing these objets in local coordinates.
Thus, again we will work in an open set of
$\mathbb{R}^m=\mathbb{R}^n\times \mathbb{R}^c$ with $n\geq 2$ and
$c\geq 1$. We also ask that $n=ss+u$.

\subsection{Jet space}
\label{sec:jets} Let $f=(f_a)_a$ be a $k$-parameter $C^{d}$-family
of $C^r$-diffeomorphisms of $\mathbb{R}^m$ with $0<d \leq r$. To
analyze
the unfolding of $f_a$ for $a\in\mathbb{I}^k$,
we will consider on $J^d_0(\mathbb{I}^k,\mathbb{R}^m)$ the map
$\widehat{f}$ \ induced
by the family $f=(f_a)_a$ and given by
$$
\widehat{f}(J_0^d(z))=J_0^d(f\circ z)=(f_a(z_a),\partial^1_a
f_a(z_a), \dots, \partial^d_a f_a(z_a))_{|_{\,a=0}}
   \ \ \text{with} \ \ z=(z_a)_a \in
   C^d(\mathbb{I}^k,\mathbb{R}^m).
$$
Here $J^d_0(\mathbb{I}^k,\mathbb{R}^m)$ denotes the $d$-th order
jet space at $a=0$, i.e., the set of equivalence classes
$J^d_0(z)$ where $z=(z_a)_a\in C^d(\mathbb{I}^k,\mathbb{R}^m)$.
The equivalent relation is defined by declaring that
$J^d_0(u)=J^d_0(v)$ if the functions $u$ and $v$ have all of their
partial derivatives equal at $a=0$ up to $d$-th order. A useful
choice of a representative for $J_0^d(z)$ is the $d$-th order
Taylor approximation of $z$ at $a=0$. This polynomial is
completely determined by the derivatives of $z$ at $a=0$, a finite
list of numbers. Therefore, it make sense to identify
$J_0^d(\mathbb{I}^k,\mathbb{R}^m)$ with $\mathbb{R}^m \times
\mathcal{J}^d(k,m)$ where
$$
\mathcal{J}^d(k,m) =
\prod_{i=1}^{d}\mathscr{L}^i_{sym}(\mathbb{R}^k,\mathbb{R}^m)
$$
and  $\mathscr{L}^i_{sym}(\mathbb{R}^k,\mathbb{R}^m)$ denotes the
space of symmetric $i$-linear maps from $\mathbb{R}^k$ to
$\mathbb{R}^m$. Hence, clearly $J_0^d(\mathbb{I}^k,\mathbb{R}^m)$
is a Euclidian vector space of dimension
$$
  \dim J_0^d(\mathbb{I}^k,\mathbb{R}^m)= m \cdot \binom{d+k}{d} = \frac{m \cdot (d+k)!}{d! \cdot
  k!}.
$$

\begin{rem}
\label{rem:parablender-regularidad} Notice that the map
$\widehat{f}$ is of class $C^{r-d}$.
\end{rem}

\begin{notation} In order to simplify notation write
$$
   J(\mathbb{R}^m) \eqdef J_0^d(\mathbb{I}^k,\mathbb{R}^m) \qquad
   \text{and} \qquad J(z)\eqdef J^d_0(z)=(z_a,\partial^1_a z_a,\dots,\partial^d_a z_a)_{|_{\,a=0}}.
$$
Sometimes, by
considering $z_a=(x_a,y_a) \in \mathbb{R}^n\times \mathbb{R}^c$,
we will split the manifold of $(d,k)$-velocities over
$\mathbb{R}^m$ (i.e., the space of $d$-jets from $\mathbb{I}^k$ to
$\mathbb{R}^m$ at $a=0$) in the form of $
J(\mathbb{R}^m)=J(\mathbb{R}^n)\times J(\mathbb{R}^c)
$~and
$$
J(z)=(J(x),J(y)) \quad  \text{where $x=(x_a)_a  \in
   C^d(\mathbb{I}^k,\mathbb{R}^n)$ and $y=(y_a)_a  \in
   C^d(\mathbb{I}^k,\mathbb{R}^c)$.}
$$
Moreover, denote by $J({\Lambda})$ the subset of $J(\mathbb{R}^*)$
of $d$-jets $J(z)$ at $a=0$ of families of points $z=(z_a)_a\in
C^{d}(\mathbb{I}^k,\mathbb{R}^*)$ such that $z_0 \in \Lambda$,
where $\Lambda\subset \mathbb{R}^*$ and $*\in \{m,n,c\}$. Also,
denote by $
\text{$\mathscr{P}_*:J(\mathbb{R}^m)\to \mathbb{R}^*$ 
the canonical projection onto $\mathbb{R}^*$ with $*\in
\{ss,u,n,c,m\}$.}$
\end{notation}


\subsection{A family of affine blenders} \label{sec:affine-blenders-family}
We will take a $C^r$-diffeomorphism  $\Phi_0$ of $\mathbb{R}^m$
locally defined as the skew-product given
in~\S\ref{sec:affine-blender}.
In particular, we have an affine $cs$-blender $\Gamma_0$ for
$\Phi_0$ in the cube $[-2,2]^{m}$  {having the family of
almost-horizonal $C^1$-discs in $\mathcal{B}=\mathsf{R}\times B$
as a superposition region}. Here $B$ is an open neighborhood of
$0$ in $D=(-2,2)^c$ satisfying
the covering property~\eqref{eq:cover} and
$\mathsf{R}=\mathsf{R}_1\cup\dots\cup\mathsf{R}_\kappa$. Now, we
will take a particular family $\Phi=(\Phi_a)_a$, unfolding
$\Phi_0$ at $a=0$. Namely, we consider $C^r$-diffeomorphisms
$\Phi_a$ locally defined in a similar way by means of
skew-products of the form
$$
   \Phi_a=F\ltimes(\phi_{1,a},\dots,\phi_{\kappa,a})
\quad  \text{on} \ \ \mathcal{U}=(\mathsf{R}_1\times
D)\cup\dots\cup (\mathsf{R}_\kappa\times D)
$$
where $\phi_{\ell,a}$ are $k$-parameter $C^d$-families of affine
$(\lambda,\beta)$-contractions on $D$ for $\nu<\lambda<\beta<1$.
That is, $\phi_\ell=(\phi_{\ell,a})_a$ is a $k$-parameter
$C^d$-family of $C^r$-diffeomorphisms $\phi_{\ell,a}$ of
$\mathbb{R}^c$ such that $\phi_{\ell,a}(\overline{D})\subset D$
and there are linear maps $T_{\ell,a}:\mathbb{R}^c\to
\mathbb{R}^c$ so that
$$
\text{$D\phi_{\ell,a}(y)=T_{\ell,a}$ for all $y\in\overline{D}$ \
\ and  \ \ $\lambda<m(T_{\ell,a})\leq \|T_{\ell,a}\|<\beta$ for
$\ell=1,\dots,\kappa$ and $a\in \mathbb{I}^k$.}
$$
Moreover, we ask that { a bounded open neighborhood $\widehat{B}$
of  the $d$-jet $0$ in $J(\mathbb{R}^c)$} such that
\begin{equation}
\label{eq:cover-jets}
 \overline{\widehat{B}} \subset \widehat{\phi}_1(\widehat{B})\cup \dots\cup
   \widehat{\phi}_\kappa(\widehat{B})
\end{equation}
where $\widehat{\phi}_\ell$ is the induced map on $J(D)$ by the
family $\phi_\ell=(\phi_{\ell,a})^{}_a$, i.e.,
\begin{equation*}
  \widehat{\phi}_\ell(J(y))=J(\phi_\ell\circ y)=
  (\phi_{\ell,a}(y_a), {\partial^1_a \phi_{\ell,a}(y_a),\dots,\partial^d_a \phi_{\ell,a}(y_a)})_{|_{\,a=0}}
\end{equation*}
with $y=(y_a)_a \in C^d(\mathbb{I}^k,\mathbb{R}^c)$ such that $y_0
\in D$.
Without restriction of generality we can assume  that
$B\times\{0\}\subset \widehat{B}$ where $B$ is the open set given
in~\eqref{eq:cover}.

On the other hand, let $\Gamma_a$ be the affine $cs$-blender
continuation of $\Gamma_0$ for $\Phi_a$. To conclude the proof we
need to prove that $\Gamma=(\Gamma_a)_a$ is a $cs$-parablender of
$\Phi=(\Phi_a)_a$ at $a=0$.

\begin{rem} \label{rem:any-parameter}
The family $\Phi=(\Phi_a)_a$ can be seen as an unfolding of
$\Phi_{a_0}$ for any $a_0 \in \mathbb{I}^k$. Since $\phi_{\ell,a}$
varies $C^d$-continuously with $a\in \mathbb{I}^k$, a similar
covering property as in~\eqref{eq:cover-jets} holds for the maps
$\widehat{\phi}_\ell=J_{a_0}^d[\phi_\ell]$. These are induced by
the families of fiber maps $\phi_{\ell}=(\phi_{\ell,a})_a$ on the
$d$-jet space $J_{a_0}^d(\mathbb{I}^k,\mathbb{R}^c)$ at $a=a_0$
for all sufficiently small parameter $a_0$. That is,
$$
  \widehat{\phi}_\ell(J(y))=
  (\phi_{\ell,a}(y_a), {\partial^1_a \phi_{\ell,a}(y_a),\dots,\partial^d_a \phi_{\ell,a}(y_a)})_{|_{\,a=a_0}}
$$
with $y=(y_a)_a \in C^d(\mathbb{I}^k,\mathbb{R}^c)$ such that
$y_{a_0} \in D$. In what follows, we will show that $\Gamma$ is a
 {$C^d$-}$cs$-parablender of $\Phi$ at $a=0$.
However, the choice of $a=0$ is only for convenience to fix an
unfolding parameter (and thus a jet space). The same argument
works to prove that $\Gamma$ is a
 {$C^d$-}$cs$-parablender of $\Phi$ at $a=a_0$ for
any $a_0$ close enough to~$0$. In fact, by continuity with respect
to the parameter, we can take an uniform open set $\mathscr{D}$ of
families of discs and an uniform neighborhood $\mathscr{U}$ of the
family $\Phi$ for all $a_0$ close to~$0$. Therefore
$\Gamma=(\Gamma_a)_a$ will be, up to scaling the parametrization,
a  {$C^d$-}$cs$-parablender of the $k$-parametric family
$\Phi=(\Phi_a)_a$ at any value of the parameter $a\in
\mathbb{I}^k$.
\end{rem}

\begin{exap}
\label{exa} 
Let $\phi(t)=\lambda t$ for $t\in [-2,2]$ with $1/2<\lambda<1$.
Set
$$
    \Upsilon=\{\iota=(\iota_1,\dots,\iota_k)\in
\{0,1,\dots,d\}^k \ \ \text{with}  \ \
|\iota|=\iota_1+\dots+\iota_k \leq d\} \ \ \text{and} \ \
\Delta=(1-\lambda)\cdot\{-1,+1\}^\Upsilon.
$$
Each $\delta \in \Delta$ is seen as a function which maps $\iota
\in \Upsilon$ to $\delta(\iota)\in \{-(1-\lambda),+(1-\lambda)\}$.
Take
\begin{align*}
    \phi_{\delta,a}(t)= \phi(t)+   P_{\delta}(a) \quad
    \text{for \ $\delta \in \Delta$,  \  $a\in \mathbb{I}^k$ \ and \  $t \in [-2,2]$  }
\end{align*}
where
$$
   P_{\delta}(a) =\sum_{\iota \in \Upsilon} \delta(\iota) \
    a^\iota \quad \text{with \ \ $a^\iota = a^{\iota_1}_1\cdots a^{\iota_k}_k$ \ and \ $\iota=(\iota_1,\dots,\iota_k)$.}
$$
Finally, consider
\begin{align}
\label{eq:multi-index}
       \phi_{\ell,a} = \phi_{\ell_1,a}\times \dots\times \phi_{\ell_c,a}
       \quad
       \text{on \  $D$ } \ \
       \text{for any \ $\ell=(\ell_1,\dots,\ell_c)\in \Delta^c$.}
\end{align}
Here, $\kappa=\varrho^c$  where $\varrho$ is the cardinal of
$\Delta$. When there are no parameters, i.e., for $k=0$, we
recover the Example~\ref{exa-afin-blender}. Moreover,
$D\phi_{\ell,a}(y)$ is the diagonal matrix $\lambda I$ where $I$
is the identity matrix and thus it does not depend on $a$ for all
$y\in \overline{D}$. Hence,  we can rewrite~\eqref{eq:multi-index}
as
\begin{equation}
\label{eq:mult0}
   \phi_{\ell,a}(y)= \lambda y + P_\ell(a) \quad \text{for \ $y\in D$, \
   $a\in \mathbb{I}^k$ \ and \ $\ell=(\ell_1,\dots,\ell_c) \in \Delta^c$}
\end{equation}
where using multilinear algebra
$$
     P_\ell(a)= \partial^0\ell + \partial^1 \ell \cdot a +
   \frac{1}{2!} \, \partial^2\ell \cdot a^2 +\dots +
   \frac{1}{d!} \,  \partial^d\ell \cdot a^d
$$
with 
$\partial^i\ell \in
\mathscr{L}^i_\mathrm{sym}(\mathbb{R}^k,\mathbb{R}^c)$ determined
by
$$
   \partial^i\ell \cdot a^i \eqdef \partial^i\ell(a,\dots,a)=
   \sum_{|\iota|=i} \ell(\iota) \, a^\iota \quad \text{ {for all $i=0,1,\dots,d$.}}
$$
Now, we can easily compute the induced map $\widehat{\phi}_{\ell}$
on $J(\mathbb{R}^c)$.
To do this,  {consider $y=(y_a)_a \in
C^d(\mathbb{I}^k,\mathbb{R}^c)$ such that $y_0 \in D$. Denoting by
$$
J(y)=(y^{}_a,\partial^1_ay^{}_a,\dots,\partial^d_ay^{}_a)_{|_{\,a=0}}\eqdef
(\partial^0y,\partial^1y,\dots,\partial^dy),$$}
from~\eqref{eq:mult0} we get that
$$
   \partial^i_a\phi_{\ell,a}(y^{}_a)_{|_{\,a=0}} =  {\lambda\cdot (\partial^i_a y_a)_{|_{\,a=0}}} +
   \partial^i_a P_\ell(a)_{|_{\,a=0}}=\lambda \partial^i y + \partial^i\ell \quad \text{ {for all $i=0,1,\dots,d$.}}
$$
Thus,
$$
   \widehat{\phi}_\ell(\partial^0y, {\dots,\partial^d y})=
   \lambda \cdot(\partial^0y, {\dots,\partial^d y})
   +(\partial^0\ell, {\dots,\partial^d\ell}).
$$
Consequently, $\widehat{\phi}_\ell$ is the composition of a
contracting hyperbolic linear map on $J(\mathbb{R}^c)$ with a
translation by the jet
$(\partial^0\ell,\partial^1\ell,\dots,\partial^i\ell)$. Since
$\ell$ runs over $\Delta^c$, we find that the open neighborhood
$\widehat{B}=(-1,1)^{\widehat{d}_c}$ of $0$ in $J(\mathbb{R}^c)$
satisfies~\eqref{eq:cover-jets}, where $\widehat{d}_c=\dim
J(\mathbb{R}^c)$.
\end{exap}

\subsection{Parablenders in the  $C^{d,r}$-topology for $0<d<r$}
According to Remark~\ref{rem:parablender-regularidad}, in order to
construct blenders for the induced map  $\widehat{\Phi}$ by the
$C^{d,r}$-family $\Phi=(\Phi_a)_a$ we  {will} restrict our
analysis to $0<d< r$ to obtain that $\widehat{\Phi}$ is at least
$C^1$.

\subsubsection{Blender induced in the jet space}
\label{sec:blender-jets}  {Consider $z=(z_a)_a \in
C^d(\mathbb{I}^k,\mathbb{R}^{m})$ and write $z_a=(x_a,y_a)$ with
$x=(x_a)_a\in C^d(\mathbb{I}^k,\mathbb{R}^n)$ and $y=(y_a)_a\in
C^d(\mathbb{I}^k,\mathbb{R}^c)$.} For each $i=1,\dots,d$, since
$F$ is an affine map which does not depend on $a$, the partial
derivative~is
$$
\partial^i_a \Phi_a(z_a)_{|_{\,
a=0}}=(DF(x_a)\, \partial^i_a x_a, \
\partial^i_a\phi_{\ell,a}(y_a))_{|_{\,a=0}} \ \ \text{where
$z_0=(x_0,y_0)\in \mathsf{R}_\ell\times D$ for
$\ell=1,\dots,\kappa$.}
$$
Hence, the map $\widehat\Phi$ on $J(\mathbb{R}^m)$ induced by the
family $\Phi=(\Phi_a)_a$ restricted to
$J(\mathcal{U})=J(\mathsf{R})\times J(D)$ is given by the
skew-product
$$
\widehat{\Phi}=\widehat F \ltimes
(\widehat{\phi}_1,\dots,\widehat{\phi}_\kappa) \quad \text{on} \ \
J(\mathcal{U})=J(\mathsf{R}_1) \times J(D)\cup \dots \cup
J(\mathsf{R}_\kappa) \times J(D),
$$ where $\widehat F $ acts on
$J({\mathsf{R}})$ given by
$$
\widehat{F}(J(x))=(F(x_a), \  {DF(x_a)\,
\partial^1_a
   x^{}_a,\dots, DF(x_a)\,\partial^d_a
   x^{}_a)_{|_{\,a=0}}} \ \ \text{where 
   $x_0 \in \mathsf{R}$}
$$
and $\widehat{\phi}_\ell$ is the induced map on $J(D)$ by the
family $\phi_{\ell}=(\phi_{\ell,a})_a$ for $\ell=1,\dots,\kappa$.
Moreover, $\widehat F$ has a horseshoe
$\widehat{\mathsf{\Lambda}}=\mathsf{\Lambda} \times \{0\}$
with, except for multiplicity, the same eigenvalues of $F$  and~stable~index 
$$
\widehat{d}_{ss}\eqdef\mathrm{ind}^s(\widehat{\Lambda})= \dim
J(\mathbb{R}^{ss})= ss \cdot \binom{d+k}{d}.
$$
 {Also the width of the stable cone-field on
$J(\mathsf{R})$ of $\hat{F}$ is the same width $\alpha$ that the
stable cone-field of $F$ has with respect to $\mathsf{R}$.}

Similarly, $\widehat{\phi}_\ell$ on $J(D)$ has also, except
multiplicity, the same eigenvalues of $\phi_{\ell,0}$ on $D$  for
all $\ell=1,\dots, \kappa$. Thus,
$\widehat{F}|_{\widehat{\Lambda}}$ dominates the fiber dynamics
$\widehat{\phi}_1,\dots,\widehat{\phi}_\kappa$ and also by
assumption the covering property~\eqref{eq:cover-jets} holds.
Hence, according to Theorem~\ref{thmBKR}, we have a $cs$-blender
$\widehat{\Gamma}$ of
 {central dimension}  $\widehat{d}_{c}$ 
 for $\widehat{\Phi}$, where
$$\widehat{d}_{c}\eqdef \dim J(\mathbb{R}^c)=c\cdot\binom{d+k}{d}.$$
 {Additionally, the family of
$( {\alpha},\nu,\delta)$-horizonal discs in
$\widehat{\mathcal{B}}= \widehat{\mathsf{R}} \times \widehat{B}$
is a superposition region of $\widehat{\Gamma}$, where now $\delta
< \lambda \hat{L}/2$ and $0<\hat{L}\leq L$ is the Lebesgue number
of the cover~\eqref{eq:cover-jets}}. Here $\widehat{\mathsf{R}}$
is a bounded open neighborhood on $J(\mathsf{R})$ of
$\widehat{\Lambda}=\Lambda\times\{0\}$. Moreover, by construction,
$\mathscr{P}_m(\widehat{\Gamma})=\Gamma_0$ where $\mathscr{P}_m:
J(\mathbb{R}^m) \to \mathbb{R}^m$ is the canonical projection.

\subsubsection{An open set of families of discs for the
family of affine blenders} \label{sec:discos-afin-jet-blender}
Recall that $\widehat{\mathsf{R}}$ was taken as a bounded
neighborhood on $J(\mathsf{R})$ of
$\widehat{\Lambda}=\Lambda\times\{0\}$. Since
$J(\mathbb{R}^n)=J(\mathbb{R}^{ss})\times J(\mathbb{R}^u)$, there
is no loss of generality in assuming that
\begin{equation} \label{eq:hatR}
\widehat{\mathsf{R}}=\widehat{\mathsf{R}}_{ss} \times
\widehat{\mathsf{R}}_{u} \quad  \text{with  \
$\widehat{\mathsf{R}}_{ss} \subset J(\mathbb{R}^{ss})$ \ \ and \ \
$\widehat{\mathsf{R}}_{u} \subset J(\mathbb{R}^{u})$.}
\end{equation}
 In fact, we can assume that
$\widehat{\mathsf{R}}_{ss}=(-2,2)^{ss} \times B_{ {\rho}}(0)$,
where  {$B_\rho(0)$ denotes the subset of $\mathcal{J}^d(k,ss)$ so
that the symmetric $i$-linear maps have all norm less than
$\rho$.}
Notice that the closure of $\widehat{\mathsf{R}}_{ss}$
 {can be identified with $[-2,2]^{d_{ss}}\times
[-\rho,\rho]^{\widehat{d}_{ss}-d_{ss}}$}. Hence, without loss of
generality, this set can be used to parameterize the
$(\textcolor{blue}{\alpha,}\nu,\delta)$-horizontal
$\widehat{d}_{ss}$-dimensional discs in
$\widehat{\mathcal{B}}=\widehat{\mathsf{R}}\times\widehat{B}$.


Consider a $( {\alpha,}\nu,\delta)$-horizontal $ss$-dimensional
$C^r$-disc $\mathcal{H}_0$ in $\mathcal{B}=\mathsf{R}\times B$.
Take the $k$-parametric constant family associated with
$\mathcal{H}_0$ given by
$$
\mathcal{H}=(\mathcal{H}_a)_a \quad \text{where \ \
$\mathcal{H}_a=\mathcal{H}_0$~for~all~$a\in \mathbb{I}^k$}.
$$

\begin{lem} \label{lem:disc-jets}
The set $\widehat{\mathcal{H}}$ in $J(\mathbb{R}^m)$ parameterized
by
$$
  \widehat{\mathcal{H}}(J(\xi))
  =J(\mathcal{H}\circ \xi)
  \quad \text{for $\xi=(\xi_a)_a\in C^d(\mathbb{I}^k,\mathbb{R}^{ss})$ with
$J(\xi) \in  \overline{\widehat{\mathsf{R}}_{ss}}$. }
$$
is a $( {\alpha,}\nu,\delta)$-horizontal
$\widehat{d}_{ss}$-dimensional $C^1$-disc in
$\widehat{\mathcal{B}}=\widehat{\mathsf{R}}\times \widehat{B}$ for
any   {$\alpha>0$ large enough} and $\nu>0$ small enough.
\end{lem}
\begin{proof}
According to Definition~\ref{def:almost-horizontal-disc}, we need
to show that $\widehat{h}=\mathscr{P}\circ \widehat{\mathcal{H}}$
is $\delta$-close in the $C^0$-topology to a constant function on
$\widehat{B}$ where $\mathscr{P}$ the canonical projection on the
central coordinate, i.e., onto $J(\mathbb{R}^c)$. Moreover, we
also need to show that $\widehat{h}$ is $C^1$-dominated by a
constant $C>0$ so that $C\nu < \delta$  {and
$\|D\widehat{g}\|_\infty \leq \alpha$ where $\widehat{g}$ is the
$J(\mathbb{R}^u)$-coordinate of $\widehat{\mathcal{H}}$}. Since
$\widehat{h}$  {and $\widehat{g}$ are} of class $C^{r-d}$ with
$r>d$, then $C=\|D\widehat{h}\|_{\infty} <\infty$
 {and $D=\|D\widehat{g}\|_\infty<\infty$} over the
closure of $\widehat{\mathsf{R}}_{ss}$. Thus taking $\nu>0$ small
enough  {and $\alpha>0$ large enough}, we can guarantee that $C\nu
<\delta$  {and $D\leq \alpha$}. So, we only need to prove that
there is a point $\widehat{y}\in \widehat{B}$ such that
$$
d(\widehat{h}(J(\xi) {)},\widehat{y})<\delta  \quad \text{for all
$J(\xi) \in \overline{\widehat{\mathsf{R}}_{ss}}$.}
$$
Since  $\mathcal{H}=(\mathcal{H}_a)_a$ is a constant family of
discs,
then~$\widehat{h}=\widehat{h}_0$~where~$h_0=\mathscr{P}\circ
\mathcal{H}_0$~and
$$
  \widehat{h}_0(J(\xi))=
  (h_0(\xi_a), {\partial_a^1h_0(\xi_a),\dots,\partial_a^d
  h_0(\xi_a)})_{|_{\,a=0}}
  \quad \text{for $\xi=(\xi_a)_a\in C^d(\mathbb{I}^k,\mathbb{R}^{ss})$ with
$J(\xi) \in  \overline{\widehat{\mathsf{R}}_{ss}}$. }
$$
Moreover, as $\mathcal{H}_0$ is a $(
{\alpha,}\nu,\delta)$-horizontal disc in $\mathcal{B}$, there is
$y\in B$ such that $d(h_0(x),y)<\delta$ for all $x \in
[-2,2]^{ss}$. Set $\widehat{y}=(y,0)\in \widehat{B}$. For any
$x\in [-2,2]^{ss}$, we consider $\xi=(\xi_a)_a$ given by $\xi_a=x$
for all $a\in \mathbb{I}^k$. Then $J(\xi)=(x,0)$ and
 {since $h_0(\xi_a)=h_0(x)$ for all $a\in
\mathbb{I}^k$ we have $\widehat{h}_0(J(\xi))=(h_0(x),0)$.}
Therefore $d(\widehat{h}(J(\xi)),\widehat{y})<\delta$ for all
$J(\xi) \in [-2,2]^{ss}\times \{0\}$. By continuity, and since $
{\rho}>0$ can be taken arbitrarily small, it follows that
$d(\widehat{h}(J(\xi)),\widehat{y})<\delta$ for all $J(\xi)$ in
the closure of $\widehat{\mathsf{R}}_{ss}=(-2,2)^{ss}\times B_{
{\rho}}(0)$. This completes the proof of the lemma.
\end{proof}

\begin{rem} If $\mathcal{H}_0$ is a horizontal $ss$-dimensional $C^r$-disc in
$\mathcal{B}$ then $\widehat{\mathcal{H}}$ is also a horizontal
$\widehat{d}_{ss}$-dimensional $C^1$-disc in
$\widehat{\mathcal{B}}$. Thus, in this case, we do not need a
strong contraction for the dynamics on the base. It is only
required the domination assumption $\nu <\lambda$.
\end{rem}

Since being an almost-horizontal $C^1$-disc is an open property,
any small enough $C^{d,r}$-perturbation
$\mathcal{D}=(\mathcal{D}_a)_a$ of $\mathcal{H}=(\mathcal{H}_a)_a$
still provides an almost-horizontal $\widehat{d}_{ss}$-dimensional
$C^1$-disc $\widehat{\mathcal{D}}$ in $\widehat{\mathcal{B}}$
close to $\widehat{\mathcal{H}}$ given by
$$
  \widehat{\mathcal{D}}(J(\xi))=J(\mathcal{D}\circ \xi)
  \quad
  \text{for $\xi=(\xi_a)_a \in C^d(\mathbb{I}^k,\mathbb{R}^{ss})$ with
$J(\xi)\in \overline{\widehat{\mathsf{R}}^{ss}}$.}
$$
In fact, taking $\xi_a \in [-2,2]^{ss}$ and
$z_a=\mathcal{D}_a(\xi_a)$ for all $a\in \mathbb{I}^k$, it is not
difficult to see that the image of this embedding is given by
$$
  \widehat{\mathcal{D}}=\{ \, J(z)\in J(\mathbb{R}^m):
  z=(z_a)\in C^d(\mathbb{I}^k,\mathbb{R}^m)
   \ \ \text{with} \ \ z_a\in \mathcal{D}_a \ \
   \text{for all $a\in \mathbb{I}^k$} \}\cap \overline{\widehat{\mathcal{B}}}.
$$
In this way, we take $\mathscr{D}=\mathscr{D}({\mathcal{H}})$, a
small enough $C^{d,r}$-neighborhood of $\mathcal{H}$.

\begin{rem} The superposition region $\mathscr{D}$ of an affine $cs$-parablender
 contains the open set of
\emph{almost-constant $k$-parameter $C^d$-families of
almost-horizontal $ss$-dimensional $C^r$-discs in $\mathcal{B}$.}
\end{rem}

\subsubsection{Parablenders from blenders in the jet space} \label{sec:parablender-from-blender} We will
get that $\Gamma=(\Gamma_a)_a$ is a $cs$-parablender of
$\Phi=(\Phi_a)_a$ as a consequence of the following general
result. 

\begin{prop} \label{prop:parablender}
Let $\Gamma_0$ be a $cs$-blender of  {central dimension} $c$
 {and strong stable dimension
$d_{ss}=\mathrm{ind}^s(\Gamma_0)-c$} of a $C^r$-diffeomorphism
$f_0$ of a manifold~$\mathcal{M}$. Consider a $k$-parameter
$C^d$-family $f=(f_a)_a$ unfolding $f_0$ at $a=0$ such that the
induced map $\widehat{f}$ on the manifold of $(d,k)$-velocities
$J(\mathcal{M})=J^d_0(\mathbb{I}^k,\mathcal{M})$ over
$\mathcal{M}$,  {which is} given by
$$
  \widehat{f}(J(z))=J(f\circ z)=(f_a(z_a), {\partial^1_a f_a(z_a),\dots,\partial^d_a f_a(z_a)})_{|_{\,a=0}}
   \ \ \text{with} \ \ z=(z_a)_a \in
   C^d(\mathbb{I}^k,\mathcal{M}),
$$
has a $cs$-blender $\widehat{\Gamma}$  satisfying the following
assumptions:
\begin{enumerate}[itemsep=0.1cm]
\item $\widehat{\Gamma}$ projects on $\mathcal{M}$ onto $\Gamma_0$;
\item there is a $k$-parameter $C^d$-family $\mathcal{H}=(\mathcal{H}_a)_a$ of
$d_{ss}$-dimensional $C^r$-embedded discs $\mathcal{H}_a$ into
$\mathcal{M}$ so that $\widehat{\mathcal{H}}_0 \in
\widehat{\mathscr{D}}$, where
$\widehat{\mathcal{H}}_0$ is 
contained in
$$
  \widehat{\mathcal{H}}=\{ \, J(z)\in J(\mathcal{M}):
  z=(z_a)\in C^d(\mathbb{I}^k,\mathcal{M})
   \ \ \text{with} \ \ z_a\in \mathcal{H}_a \ \
   \text{for all $a\in \mathbb{I}^k$} \}
$$
and $\widehat{\mathscr{D}}$ is  {a} superposition region of the
blender $\widehat{\Gamma}$.
\end{enumerate}
Then $\Gamma=(\Gamma_a)_a$ is a
 {$C^d$-}$cs$-parablender  {at $a=0$}
of  {central dimension $c$} for $f$, where $\Gamma_a$ is the
continuation of $\Gamma_0$ for~$f_a$.
\end{prop}

\begin{proof}
First of all, we will provide the open set of embedded discs. To
do this, similarly as in \S\ref{sec:discos-afin-jet-blender}, we
take a small $C^{d,r}$-neighborhood
$\mathscr{D}=\mathscr{D}(\mathcal{H})$ of the family
$\mathcal{H}$, so that any family $\mathcal{D}$ in $\mathscr{D}$
still gives a disc $\widehat{\mathcal{D}}_0 \in
\widehat{\mathscr{D}}$ contained in $\widehat{\mathcal{D}}$.

Next, we will construct the open set of families of
diffeomorphisms. Consider the neighborhood $\widehat{\mathscr{U}}$
of the induced map $\widehat{f}$ coming from the definition of the
blender. Take the $C^{d,r}$-neighborhood $\mathscr{U}$ of the
family $f=(f_a)_a$, so that for every $g=(g_a)_a \in \mathscr{U}$
its induced map $\widehat{g}$ on $J(\mathcal{M})$ belongs to
$\widehat{\mathscr{U}}$.

Now, we will prove the existence of a {degenerate} unfolding
 {at $a=0$ of a tangency} between any family of
$ss$-dimensional discs $\mathcal{D}=(\mathcal{D}_a)_a \in
\mathscr{D}$ and the unstable manifold of
$\Gamma_g=(\Gamma_{a,g})_a$ for any $g=(g_a)_a \in \mathscr{U}$,
where $\Gamma_{a,g}$ is the continuation of $\Gamma_a$ for $g_a$.
Since $\widehat{\mathcal{D}}$ contains a disc
$\widehat{\mathcal{D}}_0$ in the superposition region
$\widehat{\mathscr{D}}$ of the $cs$-blender $\widehat{\Gamma}$ of
$\widehat{f}$, then
$$W^u_{loc}(\widehat{\Gamma}_g)\cap
\widehat{\mathcal{D}}_0\not =\emptyset$$ where
$\widehat{\Gamma}_g$ is the continuation of $\widehat{\Gamma}$ for
the induced map $\widehat{g}$. It is clear that
$\mathscr{P}_m(\widehat{\Gamma}_{g})=\Gamma_{0,g}$ and that
$\widehat{\Gamma}_g$ is  a hyperbolic set of $\widehat{g}$.
If $J(z)\in \widehat{\Gamma}_g$, where $z=(z_a)_a \in
C^d(\mathbb{I}^k,\mathcal{M})$, then $z_0 \in \Gamma_{0,g}$ and
the point $z_a$  must be the continuation in $\Gamma_{a,g}$ of
$z_0$ for $g_a$. Similarly,
$\mathscr{P}_m(W^u_{loc}(J(z)))=W^u_{loc}(z_0)$ and if
$\widehat{x}\in W^u_{loc}(J(z))$ then
$$
\text{$x=(x_a)_a\in C^d(\mathbb{I}^k,\mathcal{M})$ \  so that \
$x_a\in W^u_{loc}(z_a)$ for all $a\in \mathbb{I}^k$ and
$J(x)=\widehat{x}$.}
$$

In summary, we can find a point $\widehat{q} \in
W^{u}_{loc}(\widehat{\Gamma}_g) \cap \widehat{\mathcal{D}}_0$.
Since $\widehat{q}$ belongs to the local unstable manifold of
$\widehat{\Gamma}_{g}$ there are functions $x=(x_a)_a,
z=(z_a)_a\in C^d(\mathbb{I}^k,\mathcal{M})$ such that
$$
\text{$x_a\in W^u_{loc}(z_a)$ with $z_a\in\Gamma_{a,g}$ for all
$a\in \mathbb{I}^k$ and $\widehat{q}=J(x)$.}
$$
On the other hand, since $\widehat{q}\in \widehat{\mathcal{D}}$,
$$\text{there is $y=(y_a)_a \in
C^d(\mathbb{I}^k,\mathcal{M})$ such that $y_a\in \mathcal{D}_a$
for all $a\in \mathbb{I}^k$ and $\widehat{q}=J(y)$.}$$ Thus
$J(x)=J(y)$. This concludes that $\mathcal{W}=(W^u_{loc}(z_a))_a$
and $\mathcal{D}=(\mathcal{D}_a)_a$  {has a tangency} at $a=0$
which unfolds $C^d$-{degenerately}. Therefore
$\Gamma=(\Gamma_a)_a$ is a  {$C^d$-}$cs$-parablender
 {at $a=0$} of  {central dimension $c
\geq 1$} for $f=(f_a)_a$ and we complete the proof of the
proposition.
\end{proof}


\subsubsection{Proof of Theorem~\ref{thm-parablender}} Take the
family of $cs$-blenders $\Gamma=(\Gamma_a)_a$ of
 {central dimension} $c>0$ of the particular family
of locally defined affine skew-products $\Phi=(\Phi_a)_a$
constructed in~\S\ref{sec:affine-blenders-family}. From
\S\ref{sec:blender-jets}, we get a $cs$-blender $\widehat{\Gamma}$
for the induced map $\widehat{\Phi}$ on $J(\mathbb{R}^m)$ which
projects in $\mathbb{R}^m$ onto $\Gamma_0$. In
\S\ref{sec:discos-afin-jet-blender} it was obtained that any
$k$-parameter constant family of horizontal discs induced a
$\widehat{d}_{ss}$-dimensional $C^1$-disc into the superposition
domain $\widehat{\mathcal{B}}$ of $\widehat{\Gamma}$. Thus, this
disc belongs to the superposition region of the induced blender.
Hence, according to Proposition~\ref{prop:parablender}, $\Gamma$
is a  {$C^d$}-$cs$-parablender of  {central dimension} $c>0$ at
$a=0$. Finally, by Remark~\ref{rem:any-parameter} and
reparameterizing if necessary, $\Phi=(\Phi_a)_a$ is $k$-parametric
$C^d$-family of $C^r$-diffeomorphisms having a
 {$C^d$-}$cs$-parablender $\Gamma=(\Gamma_a)_a$ of
 {central dimension $c\geq 1$} at any parameter.
This
completes the proof.

\section{Robust {degenerate} unfolding of heterodimensional cycles}
\label{sec:cycles} Now we will prove Theorem~\ref{thmA}. We will
consider a $C^d$-family $f=(f_a)_a$ of $C^r$-diffeomor\-phisms of
a manifold $\mathcal{M}$ parameterized by $a\in\mathbb{I}^k$ with
a  {$C^d$-}$cs$-parablender $\Gamma=(\Gamma_a)_a$ of codimension
$c\geq 1$ at any parameter. For simplicity, we will assume that
$\Gamma$ is the family of affine blenders constructed to prove
Theorem~\ref{thm-parablender}. We will assume that $f_0$ has a
heterodimensional cycle of co-index $c\geq 1$ associated with
$\Gamma_0$ and another hyperbolic periodic point $P_0$. We suppose
that $W^s(P_0)$ contains a $ss$-dimensional horizontal disc
$\mathcal{H}_0$ in the superposition domain $\mathcal{B}$ of the
$cs$-blender $\Gamma_0$ of $f_0$. Moreover, as the construction is
local, we ask that $W^s(P_a)$ contains the same disc
$\mathcal{H}_0$ for all $a\in\mathbb{I}^k$ where $P_a$ denotes the
continuation of $P_0$ for $f_a$. Hence the constant family of
discs $\mathcal{H}=(\mathcal{H}_a)$ where
$\mathcal{H}_a=\mathcal{H}_0$ for all $a\in \mathbb{I}^k$ belongs
to the open set $\mathscr{D}$ of families of embedded discs
associated with the  {$C^d$-}$cs$-parablender $\Gamma$. Thus, for
every $C^{d,r}$-close enough family $g=(g_a)_a$ of $f=(f_a)_a$ the
family of stable manifolds $W^s(P_g)=(W^s(P_{a,g}))_a$ of the
continuation $P_{a,g}$ of $P_a$ contains a family of discs
$\mathcal{D}=(D_a)_a \in \mathscr{D}$. Therefore,
 {$W^u_{loc}(\Gamma_g)$ and  $W^s(P_g)$} has a
 {tangency} at $a=0$ which unfolds
$C^d$-{degenerately}. In fact, since $\Gamma$ is a
 {$C^d$-}$cs$-parablender at any parameter, the same
argument also works for any parameter $a=a_0$. This concludes the
proof of the theorem.

\section{Robust {degenerate} unfoldings of homoclinic tangencies}
\label{sec:final} In this section we prove Theorem~\ref{thmC}. We
begin by mentioning a few words about the strategy of the proof.
Recall that to prove Theorem~\ref{thmA} we first show that any
manifold $\mathcal{M}$ of dimension at least $3$ admits a family
$f=(f_a)_a$ of diffeomorphisms $f_a$ of $\mathcal{M}$ having a
parablender at any parameter (see Theorem~\ref{thm-parablender}).
Now, to prove Theorem~\ref{thmC}, we will proceed similarly by
showing first the following result:

\begin{thm}
\label{thm-paratangencies} Any manifold $\mathcal{M}$ of dimension
$m> c+u$ admits a $k$-parameter $C^d$-family $f=(f_a)_a$ of
\mbox{$C^r$-diffeomorphisms} with { $0<d<r-1$} such that the
$k$-parameter induced $C^d$-family $f^{{G}}=(f^{{G}}_a)^{}_a$ of
$C^{r-1}$-diffeomorphisms on the $u$-th Grassmannian bundle of
$\mathcal{M}$,
$$
   f^{{G}}_a : G_u(\mathcal{M})\longrightarrow G_u(\mathcal{M}), \qquad
   f^{{G}}_a(x,E)=(f_a(x),Df_a(x)E)
$$
has a $C^d$-parablender $\Gamma^{{G}}=(\Gamma^{{G}}_a)^{}_a$ of
 {central dimension $c\geq 1$} at any parameter.
\end{thm}

As in the proof of Theorem~\ref{thm-parablender}, we will obtain a
parablender for the $k$-parameter family
$f^{{G}}=(f_a^{{G}})^{}_a$ for the induced dynamics by
constructing a blender with respect to the induced dynamics
$\widehat{f}^{{G}}$ on the manifold of $(c,k)$-velocities over
$G_u(\mathcal{M})$, i.e., on the jet space
$J^d_0(\mathbb{I}^k,G_u(\mathcal{M}))$.

\begin{rem}
\label{rem:paratangencias-regularidad} Notice that the map
$\widehat{f}^{{G}}$ is of class $C^{r-1-d}$.
\end{rem}

In what follows, we fix { $0<d<r-1$}. As in the previous section,
we will provide the proof of Theorem~\ref{thm-paratangencies}
using the local coordinates in $\mathcal{M}$. Thus, as usual, we
will work in $\mathbb{R}^m=\mathbb{R}^{ss}\times\mathbb{R}^u\times
\mathbb{R}^c$ with $u,c\geq 1$ and $n=ss+u$.

We also recall some notation
from~\S\ref{sec:affine-blenders-grasmannian}:
\begin{gather*}
 \mathsf{R}=\mathsf{R}_1\cup\dots\cup \mathsf{R}_\kappa \ \ \
 \text{and} \ \ \ \mathsf{R}^{{G}} = \mathsf{R}\times \mathcal{C}^{{G}} =
 \mathsf{R}^{{G}}_1\cup \dots \cup \mathsf{R}^{{G}}_\kappa \ \ \text{with} \ \
 \mathsf{R}_i^{{G}}=\mathsf{R}_i\times \mathcal{C}^{{G}}  \\
 \mathcal{U}=\mathsf{R}\times D = (\mathsf{R}_1\times D)\cup\dots\cup
 (\mathsf{R}_\kappa\times D)  \ \ \ \text{and} \ \ \
\mathcal{U}^{{G}}=\mathcal{U}\times \mathcal{C}^{{G}}.
\end{gather*}
Sometimes, when no confusion arises, we write $\mathcal{U}^{{G}}$
by a change of coordinates as
$$
   \mathcal{U}^{{G}}=\mathsf{R}^{{G}}\times D = (\mathsf{R}^{{G}}_1\times D)\cup\dots\cup
 (\mathsf{R}^{{G}}_\kappa\times D).
$$
\subsection{A parablender on the manifold of velocities over the Grassmanian manifold}
\label{sec:parablender-grasmmannian}
We will start considering the $k$-parameter $C^{d}$-family of
locally defined affine $C^{r}$ skew-product maps
$$
\Phi_a=F\ltimes(\phi_{1,a},\dots,\phi_{\kappa,a})  \quad \text{on
\ \ $\mathcal{U}=(\mathsf{R}_1\times D)\cup\dots\cup
(\mathsf{R}_\kappa\times D)$}
$$
introduced in~\S\ref{sec:affine-blenders-family}. For simplicity,
we assume that  there are $0<\nu<\lambda<\beta<1$ and diagonal
linear maps  $S: \mathbb{R}^{ss}\to \mathbb{R}^{ss}$, $U:
\mathbb{R}^{u} \to \mathbb{R}^u$ and
$T:\mathbb{R}^c\to\mathbb{R}^c$ such that
$$
 DF(x)=
\begin{pmatrix}
    S & 0  \\
    0 & U
\end{pmatrix}
\quad  \text{for all \ \ $ x\in \mathsf{R}$ \ \ where \ \ $\|S
\|,\ \|U^{-1}\| < \nu$}
$$
and
$$
  \text{$D\phi_{i,a}(y)=T$ \ \ for all \  $y\in \mathbb{R}^c$, \
$a\in\mathbb{I}^k$  \ and \  $i=1,\dots,\kappa$  \ \ where \ \
$\lambda <m(T)\leq \|T\|<\beta$.}
$$
Under theses assumptions, we get that $D\Phi_a(x,y)$ is the same
linear map $D\Phi$ for all $(x,y)\in \mathsf{R}\times D$ which has
$E^u_a=E^u$ with $E^u=\{0^{ss}\}\times \mathbb{R}^u \times
\{0^c\}$ as a fixed point and $\mathcal{C}^{{G}}$ as a
neighborhood of attraction for all $a\in \mathbb{I}^k$. Thus,
following \S\ref{sec:affine-blenders-grasmannian}, the induced
$C^{r}$-diffeomorphism $\Phi^{{G}}_a$ of $\Phi_a$ on the
Grassmannian manifold $G_u(\mathbb{R}^m)$ is given by
$$
\Phi_a^{{G}} = F^{{G}} \ltimes (\phi_{1,a},\dots,\phi_{\kappa,a})
\quad
    \text{on \  \
    $\mathcal{U}^{{G}}=(\mathsf{R}^{{G}}_1 \times D)\cup\dots\cup(\mathsf{R}^{{G}}_\kappa\times
    D$)}
$$
where $F^{{G}}=F \times D\Phi$ on
$\mathsf{R}^{{G}}=\mathsf{R}\times \mathcal{C}^{{G}}$. Moreover,
for each $a\in \mathbb{I}^k$ we have a $cs$-blender
$\Gamma^{{G}}_a=\Gamma_a \times \{E^u\}$ of
 {central dimension $c\geq 1$} where $\Gamma_a$ is
the $cs$-blender of $\Phi_a$. Now, we will show that the family
$\Gamma^{{G}}=(\Gamma^{{G}}_a)^{}_a$ is a $cs$-parablender of
 {central dimension} $c$ at $a=0$ for
${\Phi}^{{G}}=({\Phi}^{{G}}_a)^{}_a$.


To prove this, we  {will} work with the induced $C^1$-map on
$J(G_u(\mathbb{R}^m))\eqdef J^d_0(\mathbb{I}^k,G_u(\mathbb{R}^m))$
by the family $\Phi^{{G}}=({\Phi}^{{G}}_a)^{}_a$ given by
$$
  \widehat{\Phi}^{{G}}: J(G_u(\mathbb{R}^m)) \to
  J(G_u(\mathbb{R}^m)), \quad
  \widehat{\Phi}^{{G}}(J(z))=J(\Phi^{{G}}\circ
z)
\ \ \text{where $z=(z_a)_a \in
C^d(\mathbb{I}^k,G_u(\mathbb{R}^m))$.}
$$
According to Proposition~\ref{prop:parablender} to prove that
$\Gamma^{{G}}=(\Gamma^{{G}}_a)^{}_a$ is a
 {$C^d$-}$cs$-parablender for
$\Phi^{{G}}=(\Phi^{{G}}_a)^{}_a$ at $a=0$ we need to show the
following. First, we must prove that $\widehat{\Phi}^{{G}}$ has a
$cs$-blender $\widehat{\Gamma}^{{G}}$ which projects onto
$\Gamma^{{G}}_{0}$ and after provide a particular $C^d$-family
$\mathcal{H}^{{G}}=(\mathcal{H}^{{G}}_a)_a$ of $C^{r-1}$-discs
which induce a disc $\widehat{\mathcal{H}}^{{G}}$ in the open set
of $C^1$-discs.  {This will be done in the two next sections. }

\subsubsection{Blender}
\label{sec:blender-grasma-jets} Using local coordinates
(c.f.~\cite{M80,KK00}) in the manifold of $(d,k)$-velocities over
$G_u(\mathbb{R}^m)=\mathbb{R}^m\times G(u,m)$ we can identify
$J(\mathcal{U}^{{G}}) = J(\mathsf{R}^{{G}}) \times J(D)$.  Thus,
it is not difficult to see that $\widehat{\Phi}^{{G}}$ restricted
to $J(\mathcal{U}^{{G}})$ can be written as a skew-product map
$$
\widehat{\Phi}^{{G}}=\widehat{F}^{{G}} \ltimes
(\widehat{\phi}_1,\dots,\widehat{\phi}_\kappa) \quad
    \text{on \  \
$J(\mathcal{U}^{{G}})=J(\mathsf{R}^{{G}}_1) \times
   J(D) \cup \dots \cup J(\mathsf{R}^{{G}}_\kappa)
   \times
   J(D)$}
$$
where
$\widehat{F}^{{G}}$ is the induced map on $J(\mathsf{R}^{{G}})$ by
the map $F^{{G}}$ and $\widehat{\phi}_\ell$ are the induced maps
on $J(D)$ by the family $\phi_\ell=(\phi_{\ell,a})_a$ for
$\ell=1,\dots,\kappa$. Then, according to~\eqref{eq:cover-jets}
and Theorem~\ref{thmBKR} we only need to
prove that the base dynamics 
of $\widehat{\Phi}^{{G}}$ has a horseshoe which dominates the fiber dynamics. 
To do this, first we identify $J(\mathsf{R}^{{G}})=
  J(\mathsf{R}) \times J(\mathcal{C}^{{G}})$.  In this way,
  we write the base dynamics of $\widehat{\Phi}^{{G}}$ as a
direct product map
$$
\widehat{F}^{{G}}=\widehat{F} \times \widehat{D\Phi} \quad
\text{on \ \ $ J(\mathsf{R}) \times J(\mathcal{C}^{{G}})$}
$$
 where $\widehat F $ acts on
$J(\mathsf{R})$ by means of
$$
\widehat{F}(J(x))=
(F(x_a), {\partial^1_a F(x_a),\dots,\partial^d_a F(x_a)})_{|_{\,
a=0}}
\ \ \text{with $x=(x_a)_a \in C^d(\mathbb{I}^k,\mathbb{R}^n)$ such
that $x_0 \in \mathsf{R}$}
$$
and $\widehat{D\Phi}$ acts on $J(\mathcal{C}^{{G}})$ defined as
$$
\widehat{D\Phi}(J(E))=
(D\Phi \cdot E_a, \  {\partial_a^1(D\Phi \cdot
  E_a),\dots,\partial_a^d(D\Phi \cdot
  E_a)})_{|_{\, a=0}}
$$
with  $E=(E_a)_a \in C^d(\mathbb{I}^k,G(u,m))$ and $E_0 \in
\mathcal{C}^{{G}}$.
As in \S\ref{sec:blender-jets}, using that $F$ is an affine map
and is independent of $a$, we have that
$$
\widehat{F}(J(x))=(F(x_a), {
DF(x_a)\,\partial^1_ax_a,\dots,DF(x_a)\,\partial^d_ax_a})_{|_{\,a=0}}
$$
with $x=(x_a)_a \in C^d(\mathbb{I}^k,\mathbb{R}^n)$ such that
$x_0\in \mathsf{R}$. From here we get that $\widehat{F}$ has a
horseshoe $\widehat{\Lambda}=\Lambda \times \{0\}$ as an invariant
set. Moreover, the eigenvalues of the linear part of $\widehat{F}$
are the same as of  $DF$ at $\mathsf{R}$ and thus, as in
\S\ref{sec:blender-jets}, they dominate the fiber dynamics.

On the other hand, it is not difficult to see that
$\widehat{D\Phi}$ has the fixed point $J(E^u)$ where
$E^u=(E^u_a)^{}_a \in C^d(\mathbb{I}^k,G(u,m))$ is given by
$E^u_a=\{0^{ss}\}\times \mathbb{R}^u\times\{0^c\}$ for all $a\in
\mathbb{I}^k$. Hence
$$
   \widehat{\Lambda}^{{G}}=\widehat{\Lambda} \times \{J(E^u)\}
   \subset J(\mathsf{R}) \times
J(\mathcal{C}^{{G}})
$$
is a horseshoe for $\widehat{F}^{{G}}$  {with~stable~index 
$$
\widehat{d}^G_{ss}\eqdef\mathrm{ind}^s(\widehat{\Lambda}^G)= \dim
J(\mathbb{R}^{ss})+\dim J(G(u,m)).
$$
Also, analogously with~\S\ref{sec:blender-jets}
and~\S\ref{sec:affine-blenders-grasmannian}, the width of the
stable cone-field of $\widehat{F}^G$ on $J(\mathsf{R}^G)$
coincides with the width $\alpha$ of the stable cone-field of $F$ on
$\mathsf{R}$.}

Now we need to prove that $\widehat{F}^{{G}}$ restricted to
$\widehat{\Lambda}^{{G}}$ dominates $\widehat{\phi}_1, \dots,
\widehat{\phi}_\kappa$. Since, $\widehat{F}$ dominates the fiber
dynamics, it suffices to show that $\widehat{D\Phi}$ at $J(E^u)$
also dominates $\widehat{\phi}_1, \dots, \widehat{\phi}_\kappa$.
In local coordinates around $E^u$ we can write
\begin{equation} \label{eq:DPhi-E}
  D\Phi \cdot E \equiv P \, e + O(e^2), \qquad  E\in \mathcal{C}^{{G}}, \ \
   e \in \mathbb{R}^{d_u} \ \ \text{and} \ \  E\equiv e  \ \
   \text{with}  \ \ E^u \equiv 0
\end{equation}
being $P$ a diagonal $(d_u\times d_u)$-matrix whose eigenvalues
are dominated by $\beta \nu< 1$ and $d_u=\dim G(u,m)=u(m-u)$.
Similarly, we can identify $J(E) \equiv (e, {\partial^1
e,\dots,\partial^d e}) \in \mathbb{R}^{d_u}\times
\mathcal{J}^d(c,d_u)$ and take as a representative of $J(E)$ the
function $E=(E_a)_a$ given by
\begin{equation} \label{eq:Ea}
  E_a \equiv e+
\partial e \cdot a + \frac{1}{2!} \ \partial^2 e \cdot a^2 + \dots +
\frac{1}{d!} \ \partial^d e \cdot a^d.
\end{equation}
Substituting~\eqref{eq:Ea} into~\eqref{eq:DPhi-E}, in local
coordinates we have that
\begin{equation}
\label{eq2}
\partial_a^i(D\Phi \cdot
  E_a)_{|_{\, a=0}} \equiv   \ P \, \partial^i e + O(2) \quad
   {\text{for $i=1,\dots,d$}}
\end{equation}
where $O(2)$ is a function that envolves the products of
$\partial^s e \cdot \partial^te$ with $s+t=2$. In local
coordinates $J(E^u)\equiv(e,\partial^1 e,\dots,\partial^de)=0$ is
a fixed point of $\widehat{D\Phi}$. Moreover, from~\eqref{eq2} the
linear part at this point is given by a triangular matrix whose
diagonal elements are the eigenvalues of $P$. Since these
eigenvalues are dominated by $\beta\nu <\nu < \lambda$, then
$\widehat{D\Phi}$ dominates the fiber dynamics. This concludes the
proof of the existence of a $cs$-blender $\widehat{\Gamma}^{{G}}$
of $\widehat{\Phi}^{{G}}$ projecting on
$\Gamma^{{G}}_0=\Gamma_0\times \{E^u\}$.
 {Moreover, the family of
$(\alpha,\nu,\delta)$-horizontal $C^1$-discs in
$\widehat{\mathcal{B}}^{{G}}=\widehat{\mathsf{R}}^{{G}}\times
\widehat{B}$  is a superposition region of $\widehat{\Gamma}^G$,
where $0<\delta<\hat{L}\lambda/2$}.



\subsubsection{Discs on the manifold of velocities induced by
folding manifolds} \label{sec:disc-folding-manifold-jets} Let
$\widehat{\mathcal{B}}^{{G}}=\widehat{\mathsf{R}}^{{G}}\times
\widehat{B}$ be the superposition domain of the blender
$\widehat{\Gamma}^{{G}}$ where $\widehat{\mathsf{R}}^{{G}}$ is a
neighborhood on $J(\mathsf{R}^{{G}})=J(\mathsf{R})\times
J(\mathcal{C}^{{G}})$ of $\widehat{\Lambda}^{{G}}$. Similar as
in~\S\ref{sec:discos-afin-jet-blender}, since $\mathsf{R}\subset
\mathbb{R}^n=\mathbb{R}^{ss}\times \mathbb{R}^u$ and
$\mathcal{C}^{{G}} \subset G(u,m)$ we can take
$$
\widehat{\mathsf{R}}^{{G}}=\widehat{\mathsf{R}}_{ss} \times
\widehat{\mathsf{R}}_{u} \times \widehat{\mathsf{R}}_{G} \quad
\text{with \quad  $\widehat{\mathsf{R}}_{ss} \subset
J(\mathbb{R}^{ss})$, \ \ $\widehat{\mathsf{R}}_{u} \subset
J(\mathbb{R}^u)$ \ \ and \ \ $\widehat{\mathsf{R}}_{G} \subset
J(G(u,m))$.}
$$
In fact, we have that ${\widehat{\mathsf{R}}_{G}}$ can be taken as
an arbitrarily small neighborhood in $J(G(u,m))$ of
$J(E^u)\equiv(e, {\partial^1e,\dots,\partial^de})=0$ and
$\widehat{\mathsf{R}}^{ss}=(-2,2)^{ss} \times B_{ {\rho}}(0)$.
Again here denotes  {$B_\rho(0)$ denotes the subset of
$\mathcal{J}^d(k,ss)$ so that the symmetric $i$-linear maps have
all norm less than $\rho$.}

Now, fix a $( {\alpha,}\nu,\delta)$-folding $C^r$-manifold
$\mathcal{S}_0$ with respect to
$\mathcal{B}^{{G}}=\mathcal{B}\times \mathcal{C}^{{G}}$, where
$\mathcal{B}=\mathsf{R}\times B$ is the superposition domain of
the blender $\Gamma_0$ of $\Phi_0$. Consider the $k$-parametric
constant family of $( {\alpha,}\nu,\delta)$-folding
$C^r$-manifolds associated with $\mathcal{S}_0$ given by
$$
  \mathcal{S}=(\mathcal{S}_a)_a \quad \text{where \ \
   $\mathcal{S}_a=\mathcal{S}_0$ for all $a\in \mathbb{I}^k$}.
$$
According to Lemma~\ref{lem-induced-folding-manifold}, the set
$$
\mathcal{H}^{{G}}_0=\mathcal{S}^{{G}}_0 \cap
\overline{\mathcal{B}^{{G}}} \quad \text{with} \quad
  \mathcal{S}^{{G}}_0=\{ (z,E): z\in \mathcal{S}_0, \  E \in G(u,m)
    \ \
    \text{and} \  \dim E \cap T_z\mathcal{S}_0=u  \}
$$
is an almost-horizonal $d^{{G}}_{ss}$-dimensional $C^{r-1}$-disc
in $\mathcal{B}^{{G}}$ where $d^{{G}}_{ss}=ss+u(m-u)$. Hence, the
constant family of folding $C^r$-manifolds
$\mathcal{S}=(\mathcal{S}_a)_a$ induces a constant family
$\mathcal{H}^{{G}}=(\mathcal{H}^{{G}}_a)^{}_a$ of $C^r$-discs in
$\mathcal{B}^{{G}}$ given by
$\mathcal{H}^{{G}}_a=\mathcal{H}^{{G}}_0$ for all $a\in
\mathbb{I}^k$. By means of a similar argument as in
Lemma~\ref{lem:disc-jets} we obtain the following:

\begin{lem}
The set $\widehat{\mathcal{H}}^{{G}}$ in $J(G_u(\mathbb{R}^m))$
parameterized by
$$
  \widehat{\mathcal{H}}^{{G}}(J(\xi))=J(\mathcal{H}^{{G}}_a\circ
  \xi)=(\mathcal{H}^{{G}}_a(\xi_a), {
  \partial_a^1\mathcal{H}^{{G}}_a(\xi_a),\dots,\partial_a^d\mathcal{H}^{{G}}_a(\xi_a)})_{|_{\,a=0}}
$$
for $\xi=(\xi_a)_a\in C^d(\mathbb{R}^k,\mathbb{R}^{ss}\times
G(u,m))$ with $J(\xi)$ belongs to the closure of $
\widehat{\mathsf{R}}_{ss}\times \widehat{\mathsf{R}}_{G}$, is  a
$( {\alpha,}\nu,\delta)$-horizontal
$\widehat{d}_{ss}^{{G}}$-dimensional $C^1$-disc in
$\widehat{\mathcal{B}}^{{G}}=\widehat{\mathsf{R}}^{{G}}\times
\widehat{B}$ for any  {$\alpha>0$ large enough} and $\nu>0$ small
enough where
$$\widehat{d}^{{G}}_{ss} \eqdef
\mathrm{ind}^{s}(\widehat{\Lambda}^{{G}}) = \dim
J(\mathbb{R}^{ss}) + \dim J(G(u,m)).$$
%
\end{lem}
\begin{proof}
Since $d<r-1$, it is straight forward that
$\widehat{\mathcal{H}}^{{G}}$ is a
$\widehat{d}_{ss}^{{G}}$-dimensional $C^1$-disc in
$\widehat{\mathcal{B}}^{{G}}$. Let $\widehat{h}^{{G}}$
 {and $\widehat{g}\ ^G$ be, respectively, the
$J(\mathbb{R}^c)$-coordinate and $J(\mathbb{R}^u)$-coordinate of
$\widehat{\mathcal{H}}^{{G}}$ which correspond with the central
and unstable coordinates of the disc.}  In order to prove that
$\widehat{\mathcal{H}}^{{G}}$ is a $(
{\alpha,}\nu,\delta)$-horizontal disc, notice that
$$
C=\|D\widehat{h}^{{G}}\|_{\infty} <\infty \quad
 {\text{and} \quad D=\|D\widehat{g}\ ^G\|_\infty
<\infty} \quad  \text{over the closure of
$\widehat{\mathsf{R}}_{ss}\times \widehat{\mathsf{R}}_G$.}
$$
Hence, by taking  {$\alpha>0$ large enough and} $\nu>0$ small
enough we can always guarantee that $C\nu <\delta$
 {and $D\leq \alpha$}. Thus, we only need to show
that there is a point $\widehat{y}\in \widehat{B}$ such that
$$
d(\widehat{h}^{{G}}(J(\xi),\widehat{y})<\delta \quad \text{for all
$J(\xi^{{G}})=(J(\xi),J(E)) \in
\overline{\widehat{\mathsf{R}}_{ss}}\times
\overline{\widehat{\mathsf{R}}_G}$}.
$$
Since $\mathcal{H}^{{G}}=(\mathcal{H}^{{G}}_a)^{}_a$ is a constant
family of discs
then~$\widehat{h}^{{G}}=\widehat{h}^{{G}}_0$~where~$h^{{G}}_0=\mathscr{P}\circ
\mathcal{H}^{{G}}_0$~and
$$
  \widehat{h}^{{G}}_0(J(\xi^{{G}}))=J(h^{{G}}_0\circ
  \xi^{{G}})
  \quad \text{for \ \ $\xi^{{G}}=(\xi^{{G}}_a)_a\in C^d(\mathbb{R}^k,\mathbb{R}^{ss}\times
  G(u,m))$}
$$
with $J(\xi^{{G}})$ belonging to the closure of
$\widehat{\mathsf{R}}_{ss}\times \widehat{\mathsf{R}}_G$. The same
computation as in Lemma~\ref{lem:disc-jets} proves that
$$
\widehat{h}^{{G}}_0(J(\xi^{{G}}))=(h^{{G}}_0(\xi^{{G}}_0),0) \in
J(\mathbb{R}^c) \quad \text{for all
$J(\xi^{{G}})
\in
([-2,2]^{ss}\times \{E^u\}) \times \{0\}$.}
$$
Hence $d(\widehat{h}^{{G}}(J(\xi^{{G}})),\widehat{y})<\delta$ for
all $J(\xi^{{G}}) \in ([-2,2]^{ss}\times \{E^u\})\times \{0\}$,
where $\widehat{y}=(y,0)\in \widehat{B}$ and $y\in B$ comes from
the definition of the folding $C^r$-manifold. By continuity and
since the neighborhood $\widehat{\mathsf{R}}_G$ of
$J(E^u)=(E^u,0)$ can be taken arbitrarily small, it follows
$$
d(\widehat{h}(J(\xi^{{G}})),\widehat{y})<\delta \quad \text{for
all \ \ $J(\xi^{{G}})=(J(\xi),J(E)) \in
\overline{\widehat{\mathsf{R}}_{ss}} \times
\overline{\widehat{\mathsf{R}}_G}$.}
$$
This completes the proof of the lemma.
\end{proof}


\begin{rem}
Let $\mathcal{S}_0$ be the $( {\alpha,}\nu,\delta)$-folding
$C^r$-manifold with respect to
$\mathcal{B}^{{G}}=\mathcal{B}\times \mathcal{C}^{{G}}$ introduced
in Example~\ref{exap:dobra}. Proposition~\ref{prop2:appendix} in
Appendix~\ref{appendix} proves that for any $0<\nu<\delta$ the
constant family $\mathcal{S}=(\mathcal{S}_a)_a$ of $(
{\alpha,}\nu,\delta)$-folding $C^r$-manifolds
 induces a $( {\alpha,}\nu,\delta)$-horizontal
$C^1$-disc $\widehat{\mathcal{H}}^{{G}}$ in
$\widehat{\mathcal{B}}^{{G}}$ from the constant family of
$C^{r-1}$-discs $\mathcal{H}^{{G}}=(\mathcal{H}^{{G}}_a)^{}_a$
given by
$\mathcal{H}^{{G}}_a=\mathcal{H}_0^{{G}}=\mathcal{S}^{{G}}_0\cap\overline{\mathcal{B}^{{G}}}$.
\end{rem}

The previous lemma implies that $\widehat{\mathcal{H}}^{{G}}$
belongs to  {a} superposition region of the
blender~$\widehat{\Gamma}^{{G}}$.
This completes the proof of the particular $C^d$-family
$\mathcal{H}^{{G}}=(\mathcal{H}^{{G}}_a)_a$ of $C^{r-1}$-discs.

\subsubsection{Proof of Theorem~\ref{thm-paratangencies}}
The proof will follow from Proposition~\ref{prop:parablender} and
using a similar construction as in the proof of
Theorem~\ref{thmD}. Indeed, we take a $k$-parameter family of
$cs$-blenders $\Gamma=(\Gamma_a)_a$ of  {central dimension $c\geq
1$} for a $C^d$-family $\Phi=(\Phi_a)_a$ of
$C^{r}$-diffeomorphisms of $\mathcal{M}$ locally defined as affine
skew-product maps given at the beginning
of~\S\ref{sec:parablender-grasmmannian}. As we showed
in~\S\ref{sec:blender-grasma-jets}, these maps provide a family
$\Gamma^{{G}}=(\Gamma^{{G}}_a)^{}_a$ of $cs$-blenders of
 {central dimension $c\geq 1$} for the induced
dynamics $\Phi^{{G}}_a$ on $G_u(\mathcal{M})$ and as well as a
$cs$-blender $\widehat{\Gamma}^{{G}}$ for the map
$\widehat{\Phi}^{{G}}$ on $J(G_u(\mathcal{M}))$. Similarly, as in
the proof of Theorem~\ref{thmD}, we take a $C^d$-family
 {$\mathcal{S}=(\mathcal{S}_a)_a$ of $C^2$-robust
folding $C^r$-manifolds $\mathcal{S}_a$ with respect to the
superposition domain $\mathcal{B}^{{G}}$ of $\Phi^{{G}}$ in the
sense of Remark~\ref{rem:rob-dobra}.} Moreover, we assume that the
$C^d$-family $\mathcal{S}$ induces a $C^1$-disc in the
superposition region of $\widehat{\Gamma}^{{G}}$.  This was done
in~\S\ref{sec:disc-folding-manifold-jets} by taking a constant
family of folding manifolds. Then, according to
Proposition~\ref{prop:parablender}
 we have that
$\Gamma^{{G}}=(\Gamma^{{G}}_a)^{}_a$ is a
 {$C^d$}-parablender  {at $a=0$} of
 {central dimension $c\ge 1$} of
$\Phi^{{G}}=(\Phi^{{G}}_a)^{}_a$ at $a=0$. As in
Remark~\ref{rem:any-parameter}, we can extend the result for any
parameter $a_0$ close to $a=0$. This completes the proof of
Theorem~\ref{thm-paratangencies}.

\subsection{Proof of Theorem~\ref{thmC}}  {The following} result
is,  {basically}, a consequence of
Theorem~\ref{thm-paratangencies}.

\begin{thm} \label{thm-BR19}
For any $0<d<r-1$ and $k\geq 1$, there exists a $C^d$-family
$f=(f_a)_a$ of locally defined $C^r$-diffeomorphisms of
$\mathcal{M}$ having a family of $cs$-blenders
$\Gamma=(\Gamma_a)_a$ with unstable dimension $u\geq 1$
and a family of folding manifolds $\mathcal{S}=(\mathcal{S}_a)_a$
satisfying the following:

For any $a_0\in \mathbb{I}^k$, any family $g=(g_a)_a$ close enough
to $f$ in the $C^{d,r}$-topology and any $C^{d,r}$-perturbation
$\mathcal{L}=(\mathcal{L}_a)_a$ of $\mathcal{S}$ there exists
$z=(z_a)_a\in C^d(\mathbb{I}^k,\mathcal{M})$ such that
 {
\begin{enumerate}
\item $z_a\in\Gamma_{g,a}$, where $\Gamma_{g,a}$ denotes the
continuation for $g_a$ of the blender $\Gamma_a$,
\item the family of local unstable manifolds
$\mathcal{W}=(W^u_{loc}(z_a;g_a))_a$ and $\mathcal{L}$ have a
tangency
 of dimension $u$ at $a=a_0$ which unfolds $C^d$-degenerately.
\end{enumerate}
}
%
\end{thm}

\begin{rem} \label{rem-BR19} The central dimension of the blenders $\Gamma_a$ is $c=u^2$ and the folding manifold $\mathcal{S}_a$ has dimension $ss+c$
where $ss\geq 1$ is the strong stable dimension of the blenders
$\Gamma_a$. Thus the dimension of the manifold $\mathcal{M}$ is
$m>u+u^2$.  Theorem~\ref{thmC} follows immediately from the above
result assuming that each folding manifold $\mathcal{S}_a$ is part
of the stable manifold of a point of $\Gamma_a$.
\end{rem}

\begin{proof}[Proof of Theorem~\ref{thm-BR19}]
Let $\Gamma=(\Gamma_a)_a$ be the family of $cs$-blenders of the
$C^d$-family $f=(f_a)_a$ of $C^r$-diffeomorphisms of $\mathcal{M}$
given in Theorem~\ref{thm-paratangencies}. Consider the
 {$C^d$-}$cs$-parablender
$\Gamma^{{G}}=(\Gamma^{{G}}_a)^{}_a$  {at any parameter} of
 {central dimension $c\geq 1$} for the induced
dynamics $f^{{G}}=(f^{{G}}_a)^{}_a$ in ${G}_u(\mathcal{M})$. From
the proof of Theorem~\ref{thm-paratangencies}, we have
 a $C^d$-family of $C^2$-robust folding
$C^r$-manifolds $\mathcal{S}=(\mathcal{S}_a)_a$ with respect to
the superposition domain $\mathcal{B}^{{G}}$ of $f^{{G}}$ such
$\mathcal{S}$ induces a family of $C^r$-discs in
$\mathcal{B}^{{G}}$ contained in
$\mathcal{S}^{{G}}=(\mathcal{S}^{{G}}_a)^{}_a$ where
$$
  \mathcal{S}^{{G}}_a=\{ (z,E): z\in \mathcal{S}_a, \  E \in G(u,m)
    \ \
    \text{and} \   E \subset T_z\mathcal{S}_a  \}.
$$
According to the proof of Proposition~\ref{prop:parablender}, the
open set of $k$-parameter $C^d$-families of $C^r$-discs contains a
$C^{d,r}$-neighborhood $\mathscr{D}(\mathcal{S}^{{G}})$ of
$\mathcal{S}^{{G}}=(\mathcal{S}^{{G}}_a)^{}_a$. Denote by
$\mathscr{D}(\mathcal{S})$ a $C^{d,r}$-neighborhood of
$\mathcal{S}=(\mathcal{S}_a)_a$, so that if
$\mathcal{L}=(\mathcal{L}_a)_a \in \mathscr{D}(\mathcal{S})$ then
the induced family of discs
$\mathcal{L}^{{G}}=(\mathcal{L}^{{G}}_a)^{}_a \in
\mathscr{D}(\mathcal{S}^{{G}})$. Similarly, let $\mathscr{U}(f)$
be a $C^{d,r}$-neighborhood of $f=(f_a)_a$ so that if $g=(g_a)_a
\in \mathscr{U}(f)$ then the induced family of maps
$g^{{G}}=(g^{{G}}_a)^{}_a \in \mathscr{U}(f^{{G}})$. Here,
$\mathscr{U}(f^{{G}})$ comes from the definition of a parablender
as the neighborhood of the $k$-parameter family
$f^{{G}}=(f^{{G}}_a)^{}_a$. From Definition~\ref{def:parablender}
 {applied at  $a=a_0$}, we have
$x=(x_a)_a,y=(y_a)_a,z=(z_a)_a\in
C^d(\mathbb{I}^k,G_u(\mathcal{M}))$ such that
$$
z_a \in \Gamma^{{G}}_{a,g} \ \ x_a \in W^u_{loc}(z_a) \ \
\text{and} \ \ y_a\in \mathcal{L}^{{G}}_a \ \ \text{so that} \ \
d(x_a,y_a)=o(\|a-a_0\|^{d}) \ \ \text{at $a=a_0$}
$$
where $\Gamma^{{G}}_{a,g}$ is the continuation for $g^{{G}}_a$ of
the $cs$-blender $\Gamma^{{G}}_a$ for all $a\in \mathbb{I}^k$. Set
$\mathcal{W}_a=W^u_{loc}(\tilde{z}_a)$ where
$z_a=(\tilde{z}_a,W_a) \in \Gamma_a \times G(u,m)$
 and denote
$\mathcal{W}=(\mathcal{W}_a)_a$. Hence,
$$
\text{$x_a=(\tilde{x}_a,E_a) \in W^u_{loc}(z_a)$ implying that
$E_a = T_{\tilde{x}_a}\mathcal{W}_a$ and thus $x_a \in
G_u(\mathcal{W}_a)$.}
$$
Similarly,
$$
\text{$y_a=(\tilde{y}_a,F_a) \in \mathcal{L}_a$ implies that $F_a
\subset T_{\tilde{y}_a}\mathcal{L}_a$ and thus $y_a \in
G_u(\mathcal{L}_a)$.}
$$
Therefore, by Definition~\ref{def:tangencia}, $\mathcal{L}$
 {and} $\mathcal{W}$ has a tangency of
 {dimension} $u>0$ at $a=a_0$ which unfolds
$C^d$-{degenerately}.
This completes the proof.
\end{proof}

\appendix
\addtocontents{toc}{\protect\setcounter{tocdepth}{-1}}
\renewcommand{\thesection}{A}
\renewcommand{\theequation}{A.\arabic{equation}}
\setcounter{equation}{0} \setcounter{thm}{0}
\section{Estimates for the folding manifold of Example~\ref{exap:dobra}}
\label{appendix} We consider the $(ss+c)$-dimensional
$C^\infty$-manifold of Example~\ref{exap:dobra} given by
$$
  \mathcal{S}:[-2,2]^{ss}\times [-\epsilon,\epsilon]^c \to
  \mathbb{R}^m, \ \ \mathcal{S}(x,t)=(x,(t_1,\dots,t_u),h(t))
  \in \mathbb{R}^{ss}\times\mathbb{R}^u\times \mathbb{R}^c
$$
where $t=(t_1,\dots,t_u,\dots,t_c) \in [-\epsilon,\epsilon]^c$ and
$h(t)=(h_1(t),\dots,h_c(t))$ with
\begin{align*}
    h_i(t)=  \sum_{j=0}^{u-1} t_{j+1} t_{ju+i}  \quad \text{for $i=1,\dots,u$ \quad and} \quad h_i(t)=t_i \quad \text{for $i=u+1,\dots,c$.}
\end{align*}
Let $t=t(E)$ be the $C^{r-1}$-function on $\mathcal{C}^{{G}}$
computes in Example~\ref{exap:dobra}. We have that:

\begin{prop} \label{prop1:appendix}
For every $\varepsilon>0$ suffices small there is a neighborhood
$\mathcal{C}^{{G}}$ of $E^u$ such that
$$
  C=\|Dh\|_\infty \cdot
\max\{1,\|Dt\|_\infty\} \leq 1+\varepsilon \quad \text{over
$\mathcal{C}^{{G}}$.}
$$
Thus, $\mathcal{S}$ is a
$(\textcolor{blue}{\alpha,}\nu,\delta)$-folding manifold for any
$\nu>0$ such that $(1+\varepsilon)\nu<\delta$.
\end{prop}

\begin{proof}
First of all, notice that for all $t\in [-\epsilon,\epsilon]^c$,
it holds that
$$
   \|Dh(t)\|_{\infty} \eqdef \max_{i=1,\dots, c} |
   Dh_i(t)| =\max \bigg\{ \,1, \ \max_{i=1,\dots,u}
\big|\sum_{j=0}^{u-1} t_{j+1}+t_{ju+i}
 \big|\,\bigg\}.
$$
Hence, taking $\epsilon>0$ small enough we have that
$\|Dh\|_\infty =1$. On the other hand, by Cramer's rule we get
that the solution of the linear system $At= \vec{c}$ \ is given by
$$
   t_i =  \frac{\det A_i^*}{\det A}  \qquad \text{for
   $i=1,\dots,c$}
$$
where $A^*_{i}$ is the matrix formed by replacing the $i$-th
column of $A$ by the column~vector~$\vec{c}$. In order to compute
$Dt$ we write the variable of $t$ by $E=\langle
v_{1},\dots,v_u\rangle$ with $v_k=(a_k,b_k,c_k)\in
\mathbb{R}^{ss}\times\mathbb{R}^{u}\times\mathbb{R}^c$ and then
$$
Dt=(\partial_{v_1}t,\dots,\partial_{v_u}t) \quad \text{with} \ \ \
\partial_{v_k}
t=(\partial_{a_k}t,\partial_{b_k}t,\partial_{c_k}t) \quad
\text{for $k=1,\dots,u$}.
$$
Moreover, since the matrix $A$ does not depend on the variables
$a_k$ we get that $\partial_{a_k}t=0$. In the sequel we will use
the symbol $D_{k}$ to denote any partial derivative of the
form~$\partial_{b_{k\iota}}$~or~$\partial_{c_{k\iota}}$. For each
$i=1,\dots,c$, using Jacobi's formula it follows that
$$
  D_kt_i =\frac{D_k(\det A^*_i)-\mathrm{tr}(A^{-1} \cdot D_kA ) \cdot \det A^*_i}{\det
  A}.
$$
In particular, since  $\vec{c}(E^u)=0$ then $\det A^*_i(E^u)=0$,
and  $\det A(E^u)=2$,  we obtain that
$$
D_kt_i(E^u)=\frac{1}{2} \, D\det A^*_i|_{E^u} = \frac{1}{2} \,
\mathrm{tr}(\mathrm{Adj}(A^*_i) \cdot D_kA^*_i)|_{E^u}
$$
where $\mathrm{Adj}(A^*_i)$ is the adjugate matrix of $A^*_i$.
Notice that
$$
   \mathrm{Adj}(A^*_i(E^u))=(C_{\ell j})^T \quad \text{with \ \
   $C_{\ell j} =0$ \  if \ $j\not=i$ \ \ and \ \ $C_{\ell i}=(-1)^{\ell+i}\cdot \sigma_i$  \ \ for $\ell=1,\dots,c$}
$$
where $\sigma_i=2$ if $i\not=1$ and $\sigma_i=1$ otherwise
($i=1$). From this follows that
$$
  \mathrm{tr}(\mathrm{Adj}(A^*_i) \cdot D_kA^*_i)|_{\,(E^u)}=
  (C_{1i},\dots,C_{ci})\cdot D_k\vec{c}(E^u).
$$
If $D_k$ is either, $\partial_{b_{k\iota}}$ for $\iota=1,\dots,u$
or $\partial_{c_{k\iota}}$ for $\iota=u+1,\dots,c$  then
$D_k\vec{c}=0$ and thus $D_kt_i(E^u)=0$. Otherwise,
$$
 (C_{1i},\dots,C_{ci})\cdot D_k\vec{c}(E^u) = C_{(k-1)u+\iota \, i}
 \quad \text{and hence} \quad D_kt_i(E^u)=\frac{1}{2} \,
 C_{(k-1)u+\iota \, i}.
$$
Therefore
$$
   \|Dt(E^u)\|_{\infty}=\frac{1}{2} \ \max_{i=1,\dots,c} \ \max_{k=1,\dots,u} \  \big|\sum_{\iota=1}^{u}
   C_{(k-1)u+\iota \, i} \big|= \frac{1}{2} \
   \max_{i=1,\dots,c} \  \max_{k=1,\dots,u} \
   \big|\sum_{\iota=1}^u (-1)^{i+(k-1)u+\iota}  \sigma_i \big|
    \leq 1.
$$
Since $Dt$ varies continuously with respect to $E$ we have that
$\|Dt(E)\|_\infty$ is close to $\|Dt(E^u)\|_\infty$ for any $E$
close enough to $E^u$. Thus, shrinking $\mathcal{C}^{{G}}$ if
necessary, this implies that $\|Dt\|_\infty \leq 1+\varepsilon$
over $ \mathcal{C}^{{G}}_\alpha$. Hence, $C=\|Dh\|_\infty \cdot
\max\{1,\|Dt\|_\infty\} \leq 1+\varepsilon$ for a fixed but
arbitrarily small  $\varepsilon>0$.
\end{proof}

This $(\textcolor{blue}{\alpha,}\nu,\delta)$-folding
$C^r$-manifold $\mathcal{S}$ induces a $C^{r-1}$-disc
$$
  \mathcal{H}^{{G}}: [-2,2]^{ss}\times \mathcal{C}^{{G}}_\alpha \longrightarrow
  G_u(\mathbb{R}^m), \qquad \mathcal{H}^{{G}}(x,E)=(\mathcal{S}(x,t), E)
  \in \overline{\mathcal{B}}\times \mathcal{C}^{{G}}=
  \overline{\mathcal{B}^{{G}}}
$$
Set $h^{{G}}=\mathscr{P}\circ \mathcal{H}^{{G}}$. Here
$\mathscr{P}$ denotes the standard projection onto $\mathbb{R}^c$.
We consider
$$
  \widehat{h}^{{G}}(J(\xi^{{G}}))=J(h^{{G}}\circ
  \xi^{{G}})
  \quad \text{for \ \ $\xi^{{G}}=(\xi^{{G}}_a)_a\in C^d(\mathbb{R}^k,\mathbb{R}^{ss}\times
  G(u,m))$}
$$
with
$$
J(\xi^{{G}})=((\xi_0,E_0),\partial^1\xi^{{G}},\dots,\partial^d\xi^{{G}})
\in \big([-2,2]^{ss}\times \mathcal{C}^{{G}}\big) \times
\overline{B}_\rho(0)$$ where $\overline{B}_\rho(0)$ denotes a
closed ball of radius $\rho>0$ at $0$ velocity of the jets over
$\mathbb{R}^{ss}\times G(u,m)$.

\begin{prop} \label{prop2:appendix}
For every $\varepsilon>0$ suffices small there are $\rho>0$ and a
neighborhood $\mathcal{C}^{{G}}$ of $E^u$ so that
$$
  \|D\widehat{h}^{{G}}\|_\infty \leq 1+\varepsilon \quad \text{over \ \
$  \big([-2,2]^{ss}\times \mathcal{C}^{{G}}\big) \times
\overline{B}_\rho(0)$.}
$$
Thus, $\widehat{\mathcal{H}}^{{G}}$ is a
$(\textcolor{blue}{\alpha,}\nu,\delta)$-horizontal disc for any
$\nu>0$ such that $(1+\varepsilon)\nu<\delta$.
\end{prop}
\begin{proof}
Observe that $h^{{G}}(E)=h(t(E))$. In this way,
\begin{equation}
\label{eqA1}
  \widehat{h}^{{G}}(J(E))=J(h^{{G}}\circ
  E)
   =(h(t(E_a)),\partial^1_a h(t(E_a)),\dots,\partial^d_a h(t(E_a)))_{|_{\,a=0}}
\end{equation}
for $E=(E_a)_a \in C^d(\mathbb{I}^k,G(u,m))$ with $E_0\in
\mathcal{C}^{{G}}_\alpha$. Denoting $t_a=t(E_a)$ for all $a\in
\mathbb{I}^k$, we can rewrite~\eqref{eqA1} as
$$
\widehat{h}^{{G}}(J(t))=(h(t_a),\partial^1_ah(t_a),\dots,\partial^d_ah(t_a))_{|_{\,a=0}}
\quad \text{where $t=(t_a)_a \in C^d(\mathbb{I}^k,\mathbb{R}^c)$}
$$
with $t_0$ small enough in norm. Therefore,
\begin{equation}
\label{eqA2} D\widehat{h}^{{G}}= \frac{d\widehat{h}^{{G}}}{dJ(t)}
\cdot \frac{dJ(t)}{dJ(E)}= \frac{d\widehat{h}^{{G}}}{dJ(t)} \cdot
\frac{d \widehat t}{dJ(E)}
\end{equation}
where $\widehat{t}(J(E))=J(t\circ E)$.  We want to compute
$\|D\widehat{h}^{{G}}(J(E^u))\|_{\infty}$ where $E^u=(E^u_a)^{}_a$
with $E^u_a=E^u_0=\{0^{ss}\}\times \mathbb{R}^u\times \{0^c\}$ for
all $a\in \mathbb{I}^k$. Hence $J(E^u)=(E^u_0,0)$. By a
straightforward calculation using Fa\`a di Bruno's formula, we
have
\begin{equation}
 \label{eqA3} \big\| \frac{d\widehat{t}}{dJ(E)}(J(E^u))\big \|_{\infty} =
 \|Dt(E^u_0)\|_\infty \leq 1.
\end{equation}
By means of a similar computation we can show that
\begin{equation}
 \label{eqA4} \big\| \frac{d\widehat{h}^{G}}{dJ(t)}(J(t^{u}))\big \|_{\infty} =
 \|Dh(0)\|_\infty = 1
\end{equation}
where $t^{u}=(t^{u}_a)^{}_a$ with $t^{u}_a=t(E^u_a)=t(E^u_0)=0$
for all $a\in \mathbb{I}^k$
and hence  $J(t^{u})=0$.
Finally putting together \eqref{eqA2}-\eqref{eqA4} we get that
$$
\|D\widehat{h}^{{G}}(J(E^u))\|_\infty \leq \big\|
\frac{d\widehat{h}^{{G}}}{dJ(t)}(J(t^{{G}})) \big\| \cdot \big\|
\frac{dJ(t)}{dJ(E)} (J(E^u)) \big\| \leq 1.
$$
By continuity with respect to $J(E)$, shrinking
$\mathcal{C}^{{G}}$ if necessary and taking $\rho>0$ small enough,
we have $\|D\widehat{h}^{{G}}\|_{\infty} \leq 1+\varepsilon$ over
$\big([-2,2]^{ss}\times \mathcal{C}^{{G}}_\alpha\big) \times
\overline{B}_\rho(0)$ for a fixed but arbitrarily small
$\varepsilon>0$. This completes the proof.
\end{proof}

\subsection*{Acknowledgements}
We thank P.~Berger who told us the ideais behind his proof
in~\cite{Ber17}.
We are grateful to J.~Rojas for helping us to understand
P.~Berger's articles and for the initial discussions on this
project. We also thank the anonymous referee for his comments
and pointing out some problems.

\bibliographystyle{alpha}
\bibliography{br-bib}

\def\cprime{$'$}
\begin{thebibliography}{{Ber}17b}

\bibitem[ACW17]{ACW17}
Artur Avila, Sylvain Crovisier, and Amie Wilkinson.
\newblock ${C}^1$ density of stable ergodicity.
\newblock {\em arXiv preprint arXiv:1709.04983}, 2017.

\bibitem[Asa08]{Ao08}
M.~Asaoka.
\newblock Hyperbolic sets exhibiting {$C^1$}-persistent homoclinic tangency for
  higher dimensions.
\newblock {\em Proc. Amer. Math. Soc.}, 136(2):677--686, 2008.

\bibitem[BBD16]{BBD16}
J.~Bochi, C.~Bonatti, and L.~J. D{\'i}az.
\newblock Robust criterion for the existence of nonhyperbolic ergodic measures.
\newblock {\em Communications in Mathematical Physics}, 344(3):751--795, 2016.

\bibitem[BCP16]{BCP17}
P.~{Berger}, S.~{Crovisier}, and E.~{Pujals}.
\newblock {Iterated Functions Systems, Blenders and Parablenders}.
\newblock {\em in Fractals and Related Fields III: Proc. Conf., Ile de
  Porquerolles (France), 2015 (in press)}, 2016.

\bibitem[BD96]{BD96}
C.~Bonatti and L.~J. D{\'{\i}}az.
\newblock Persistent nonhyperbolic transitive diffeomorphisms.
\newblock {\em Ann. of Math. (2)}, 143(2):357--396, 1996.

\bibitem[BD08]{BD08}
C.~Bonatti and L.~J. D{\'{\i}}az.
\newblock Robust heterodimensional cycles and {$C^1$}-generic dynamics.
\newblock {\em J. Inst. Math. Jussieu}, 7(3):469--525, 2008.

\bibitem[BD12]{BD12}
C.~Bonatti and L.~J. D{\'{\i}}az.
\newblock Abundance of {$C^1$}-homoclinic tangencies.
\newblock {\em Trans. Amer. Math. Soc.}, 264:5111--5148, 2012.

\bibitem[BDV05]{BDV05}
Christian Bonatti, Lorenzo~J. D{\'{\i}}az, and Marcelo Viana.
\newblock {\em Dynamics beyond uniform hyperbolicity}, volume 102 of {\em
  Encyclopaedia of Mathematical Sciences}.
\newblock Springer-Verlag, Berlin, 2005.
\newblock A global geometric and probabilistic perspective, Mathematical
  Physics, III.

\bibitem[Ber16]{Ber16}
P.~Berger.
\newblock Generic family with robustly infinitely many sinks.
\newblock {\em Inventiones mathematicae}, 205(1):121--172, 2016.

\bibitem[Ber17a]{Ber17}
P.~Berger.
\newblock Emergence and non-typicality of the finiteness of the attractors in
  many topologies.
\newblock {\em Proceedings of the Steklov Institute of Mathematics},
  297(1):1--27, May 2017.

\bibitem[{Ber}17b]{Ber19}
P.~{Berger}.
\newblock {Generic family displaying robustly a fast growth of the number of
  periodic points}.
\newblock {\em ArXiv e-prints}, January 2017.

\bibitem[Bie16]{biebler2016persistent}
S{\'e}bastien Biebler.
\newblock Persistent homoclinic tangencies and infinitely many sinks for
  residual sets of automorphisms of low degree in ${C}^{3}$, 2016.

\bibitem[BKR14]{BKR14}
P.~G. Barrientos, Y.~Ki, and A.~Raibekas.
\newblock Symbolic blender-horseshoes and applications.
\newblock {\em Nonlinearity}, 27(12):2805, 2014.

\bibitem[BR17]{BR17}
P.~G. {Barrientos} and A.~{Raibekas}.
\newblock {Robust cycles and tangencies of large codimension}.
\newblock {\em Nonlinearity}, 2017.

\bibitem[BR18]{BR18}
Pablo~G Barrientos and Artem Raibekas.
\newblock Robustly non-hyperbolic transitive symplectic dynamics.
\newblock {\em Discrete \& Continuous Dynamical Systems-A}, 38(12):5993, 2018.

\bibitem[DK00]{KK00}
M.~Krupka D.~Krupka.
\newblock Jets and contact elements.
\newblock In {\em Proceedings of the Seminar on Differential Geometry}, pages
  39--85. Mathematical Publications Vol. 2, 2000.

\bibitem[Duj17]{dujardin2017non}
Romain Dujardin.
\newblock Non-density of stability for holomorphic mappings on $\mathbb{P}^k$.
\newblock {\em Journal de l'{\'E}cole polytechnique{-}Math{\'e}matiques},
  4:813--843, 2017.

\bibitem[G{\v{S}}72]{GS72}
NK~Gavrilov and LP~{\v{S}}il'nikov.
\newblock On three-dimensional dynamical systems close to systems with a
  structurally unstable homoclinic curve. i.
\newblock {\em Sbornik: Mathematics}, 17(4):467--485, 1972.

\bibitem[GST08]{GST08}
Sergey~V Gonchenko, Leonid~P Shilnikov, and Dmitry~V Turaev.
\newblock On dynamical properties of multidimensional diffeomorphisms from
  newhouse regions: I.
\newblock {\em Nonlinearity}, 21(5):923, 2008.

\bibitem[GTS93]{GTS93}
S.~V. Gonchenko, D.~V. Turaev, and L.~P. Shil{\cprime}nikov.
\newblock On the existence of {N}ewhouse regions in a neighborhood of systems
  with a structurally unstable homoclinic {P}oincar\'e curve (the
  multidimensional case).
\newblock {\em Dokl. Akad. Nauk}, 329(4):404--407, 1993.

\bibitem[Mic80]{M80}
P.W. Michor.
\newblock {\em Manifolds of differentiable mappings}.
\newblock Shiva mathematics series. Shiva Pub., 1980.

\bibitem[New70]{New70}
S.~E. Newhouse.
\newblock Nondensity of axiom {A(a)} on {S2}.
\newblock In {\em Global {A}nalysis ({P}roc. {S}ympos. {P}ure {M}ath., {V}ol.
  {XIV}, {B}erkeley, {C}alif., 1968)}, pages 191--202. Amer. Math. Soc.,
  Providence, R.I., 1970.

\bibitem[New74]{New74}
S.~E. Newhouse.
\newblock Diffeomorphisms with infinitely many sinks.
\newblock {\em Topology}, 13:9--18, 1974.

\bibitem[New79]{New79}
S.~E. Newhouse.
\newblock The abundance of wild hyperbolic sets and nonsmooth stable sets for
  diffeomorphisms.
\newblock {\em Inst. Hautes \'Etudes Sci. Publ. Math.}, 50:101--151, 1979.

\bibitem[NP12]{NP12}
M.~Nassiri and E.~R. Pujals.
\newblock Robust transitivity in hamiltonian dynamics.
\newblock {\em Ann. Sci. \'Ec. Norm. Sup\'er.}, 45(2), 2012.

\bibitem[PS96]{PS95}
C.~Pugh and M.~Shub.
\newblock Stably ergodic dynamical systems and partial hyperbolicity.
\newblock In {\em InternationalConference on Dynamical Systems (Montevideo,
  1995)}, pages 182--187, Longman, Harlow, 1996. Pitman Res. Notes Math. Ser.

\bibitem[PV94]{PV94}
J.~Palis and M.~Viana.
\newblock High dimension diffeomorphisms displaying infinitely many periodic
  attractors.
\newblock {\em Ann. of Math. (2)}, 140(1):207--250, 1994.

\bibitem[Roj17]{R17}
J.~D. Rojas.
\newblock {\em Generic amilies exhibiting infinitely many non-uniform
  hyperbolic attractors for a set of total measure of parameters}.
\newblock PhD thesis, Instituto de Matem\'atica Pura e Aplicada (IMPA), 2017.

\bibitem[Rom95]{Ro95}
N.~Romero.
\newblock Persistence of homoclinic tangencies in higher dimensions.
\newblock {\em Ergodic Theory Dynam. Systems}, 15(4):735--757, 1995.

\bibitem[Taf17]{taflin2017blenders}
Johan Taflin.
\newblock Blenders near polynomial product maps of $\mathbb{C}^2$, 2017.

\end{thebibliography}

\end{document}